\documentclass[a4paper]{article}
\usepackage[T1]{fontenc}
\usepackage[utf8x]{inputenc}
\usepackage[english]{babel}
\usepackage{amsmath,amsfonts,amssymb}
\usepackage{hyperref}

\usepackage{amsthm}

\theoremstyle{definition}
\newtheorem{D}{Definition}

\theoremstyle{plain}
\newtheorem*{theorem*}{Theorem}
\newtheorem{theorem}{Theorem}[section]
\newtheorem{lemma}[theorem]{Lemma}
\newtheorem{cor}[theorem]{Corollary}
\newtheorem{prop}[theorem]{Proposition}
\newtheorem{conj}[theorem]{Conjecture}
\theoremstyle{remark}
\newtheorem{rem}[theorem]{Remark}

\newcommand{\R}{\ensuremath{\mathbb R}}

\newcommand{\s}{\ensuremath{\mathbb S}}
\newcommand{\eps}{\ensuremath{\varepsilon}}
\newcommand{\p}{\ensuremath{\frac{2n}{n-2}}}
\newcommand{\norm}[1]{\ensuremath{\left\|#1\right\|}}
\newcommand{\cutoff}{\ensuremath{\rho_{\varepsilon}}}

\newcommand{\tub}[1]{\ensuremath{\Sigma^{#1}}}
\newcommand{\Sob}[1]{\ensuremath{W^{1,#1}(X)}}
\newcommand{\Yl}[1]{\ensuremath{Y_{\ell}(#1)}}

\DeclareMathOperator{\vol}{Vol}

\author{Ilaria Mondello}
\title{The Local Yamabe Constant of Einstein Stratified Spaces}
\begin{document}
\date{}
\maketitle

\nocite{*}

\begin{abstract}
On a compact stratified space $(X,g)$ there exists a metric of constant scalar curvature in the conformal class of $g$, if the scalar curvature $S_g$ satisfies an integrability condition and if the Yamabe constant of $X$ is strictly smaller than the local Yamabe constant $\Yl{X}$, another conformal invariant introduced in the recent work of K. Akutagawa, G. Carron and R. Mazzeo. Such invariant depends on the local structure of $X$, in particular on the links, but its explicit value is not known. We are going to show that if the links satisfy a Ricci positive lower bound, then we can compute $\Yl{X}$. In order to achieve this, we prove a lower bound for the spectrum of the Laplacian, by extending a well-known theorem by Lichenrowicz, and a Sobolev inequality, inspired by a result due to D. Bakry. Furthermore, we prove the existence of an Euclidean isoperimetric inequality on particular stratified space, with one stratum of codimension 2 and cone angle bigger than $2\pi$.
\end{abstract}

\section*{Introduction}
Stratifications of topological spaces have been introduced by H. Whitney \cite{Whitney} with the basic idea of partitioning a space in simpler elements, like manifolds, which are glued together in an appropriate way. Stratified spaces have been largely studied from a topological point of view (\cite{Thom}, \cite{Whitney}). They appear, for example, in treating the stability of smooth mappings between manifolds \cite{Mather}. Moreover, they give an appropriate setting to formulate Poincaré duality for intersection homology on singular spaces \cite{GoreskyMacPherson}. 

Analysis on stratified spaces is a quite recent field of investigation, since the 80s with the works by J. Cheeger on the spectral analysis on manifolds with conical singularities or corners \cite{Cheeger}. Another interesting approach is given by R. Merlose's study of pseudo-differential operators on singular spaces.

Furthermore, stratified spaces also arise in differential geometry, for example as quotients of compact Riemannian manifolds: the American football with cone angle $2\pi/p$, for an integer $p$, is a quotient of the sphere. They appears also as limits of smooth Riemannian manifolds. Later on in this paper, we give an example of a stratified space appearing as the limit of smooth complete surfaces: the cone of angle bigger than $2\pi$ over a circle. 


We are interested in studying stratified spaces by using classical tools from Riemannian geometry and geometric analysis. In particular, we consider the Yamabe problem on a compact stratified space $X$ endowed with an iterated edge metric $g$. Let us briefly recall some of the definitions we need in the following: we mainly refer to \cite{ACM12}, \cite{ACM14} and \cite{APLMP} for a more detailed discussion. 

\subsection*{The setting: Stratified spaces}
A compact stratified space $X$ is a metric space which admits a decomposition in a finite number of strata $X_j$ 
\begin{equation*}
X\supseteq X_{n-2}\supseteq X_{n-3} \supseteq \ldots \supseteq X_j  \supseteq \ldots
\end{equation*}
such that for each $j$, $X_j \setminus X_{j-1}$ is a smooth open manifold of dimension $j$, and $\Omega=X\setminus X_{n-2}$ is dense in $X$. By assumption, there is no codimension $1$ stratum.
We denote by $\Sigma$ the singular set of $X$, i.e. $\Sigma = X \setminus \Omega$. Each connected component of $X_j\setminus X_{j-1}$ has a tubular neighbourhood $\mathcal{U}_j$ which is the total space of a smooth cone bundle. Its fibre is $C(Z_j)$, where $Z_j$ is a stratified space and it is called the \emph{link} of (connected component of) the stratum. In the following we assume for simplicity that each stratum is connected, but clearly our discussion applies to each connected component. We follow \cite{ACM14} in identifying a neighbourhood of a point $x \in X_{j}\setminus X_{j-1}$ with the cone bundle: there exists a radius $\delta_x$, a neighbourhood $\mathcal{U}_x$ of $x$ and a homeomorphism 
\begin{equation*}
\varphi_x: B^{j}(\delta_x)\times C_{\delta_x}(Z_j) \rightarrow \mathcal{U}_x
\end{equation*} 
$\varphi_x$ is a diffeomorphism between $(B^{j}(\delta_x)\times C_{\delta_x}(Z_{j, reg})) \setminus (B^{j}(\delta_x)\times \{0\})$ (where $Z_{j,reg}$ is the regular part of $Z_j$) and the regular part of $\mathcal{U}_x$, i.e. $\mathcal{U}_x \cap \Omega$. 

The simplest examples of stratified spaces are manifolds with conical singularities or with simple edges: in this last case each link $Z_j$ is compact smooth manifold. 

We define iteratively the notion of \emph{depth}. If $X$ is a smooth compact manifold, it has depth equal to 0. If $Z$ is a stratified space of depth $k$ and $X$ is a stratified space with just one stratum having as link $Z$, then the depth of $X$ is $k+1$. In general, the depth of a stratified space is the maximal depth of the links of his strata, plus one. Depth allows us to apply iterative arguments on stratified spaces and in particular to define admissible metrics on them. An iterated edge metric $g$ on $X$ is a Riemannian metric on $\Omega$ which near to each stratum $X_j$ can be written as
\begin{equation*}
g= dy^2+dx^2+x^2k_j+E
\end{equation*}
where $dy^2$ is the Euclidean metric on $\R^{j}$, $k_j$ is an iterated edge metric on the link $Z_j$, and $E$ is a perturbation decaying as $x^{\gamma}$ for some $\gamma>0$.

It is possible to define the Sobolev space $\Sob{2}$ as the closure of Lipschitz functions on $X$ with the usual Sobolev norm; the set $C^1_0(\Omega)$ is dense in $W^{1,2}(X)$. Moreover, it is proved in \cite{ACM12} that the continuous Sobolev embedding of $W^{1,2}(X)$ in $L^{\p}(X)$ holds. 

\subsection*{The Yamabe Problem on Stratified Spaces}
A classical problem in geometric analysis was posed in the 60s by H. Yamabe: given a compact smooth manifold $(M^n,g)$, $n\geq 3$, is it possible to find a conformal metric $\tilde{g}$ such that
\begin{equation*}
\tilde{g}=u^{\frac{4}{n-2}}g
\end{equation*}
for some positive smooth function $u$, and the scalar curvature $S_{\tilde{g}}$ of $\tilde{g}$ is constant? 
The results by Trudinger \cite{Trudinger}, T.Aubin \cite{Aubin}, and finally R. Schoen \cite{Schoen84}, led to a positive answer. The solution of the problem strongly depends on finding a smooth positive function attaining the Yamabe constant:
\begin{equation*}
Y(M^n,[g])=\inf_{\underset{u>0}{u \in W^{1,2}(M)}}\frac{\int_M (|du|^2+a_nS_gu^2)dv_g}{\norm{u}_{\p}^2},
\quad a_n=\frac{(n-2)}{4(n-1)}.
\end{equation*}
In particular, T.Aubin proved that for any smooth compact manifold $Y(M^n,[g])$ is smaller or equal than the Yamabe constant $Y_n$ of the standard sphere $\s^n$. Furthermore, when the inequality is strict, then there exists a minimizer attaining $Y(M,[g])$. He also proved that for $n \geq 6$ and $(M^n,g)$ not locally conformally flat, the strict inequality holds. In the other cases, his works together with the proof of the positive mass theorem (\cite{SchoenYau79}, \cite{SchoenYau81}, \cite{SchoenYau88}), allowed Schoen to prove that either the inequality is strict, or $(M^n,[g])$ is conformal to the standard sphere.
 
In \cite{ACM12}, the authors considered the analogous problem on compact stratified spaces. Since the Sobolev embedding holds on $(X,g)$, it is possible to define the Yamabe constant of $X$ in the same way as in the smooth case. Nevertheless, this constant may not be finite if there is not any control the scalar curvature $S_g$: we assume that $S_g$ satisfies an integrability condition, i.e. $S_g \in L^{q}(X)$ for $q >n/2$. Moreover, \cite{ACM12} introduced another conformal invariant, the local Yamabe constant. They first define the Yamabe constant of an open ball (or set) of $X$: it will be equal to
\begin{equation*}
Y(B(p,r))=\inf\left\{\int_X (|du|^2+a_nS_gu^2)dv_g, u\in W^{1,2}_0(\Omega \cap B(p,r)), \norm{u}_{\p}=1\right\}.
\end{equation*}
Then the local Yamabe constant of $X$ is defined as follows:
\begin{equation*}
\Yl{X}= \inf_{p \in X} \lim_{r\rightarrow 0}Y(B(p,r)).
\end{equation*}
When $p$ belongs to the regular set, the limit as $r$ goes to zero of $Y(B(p,r))$ is clearly equal to $Y_n$, so that by definition $\Yl{X}\leq Y_n$. Furthermore, thanks to the local geometry of stratified spaces, the local Yamabe constant turns out to be equal to:
\begin{equation}
\label{lYc}
\Yl{X}=\min_{j=0 \ldots n} \inf_{p\in X_j\setminus X_{j-1}} \left\{ Y(\R^{j}\times C(Z^j), [dy^2+dx^2+x^2(k_j)_p])\right\}
\end{equation}
The local Yamabe constant plays the same role as $Y_n$ in the classical Yamabe problem. It is shown in \cite{ACM12} that if the Yamabe constant of a compact stratified space $X$ is strictly smaller than its local Yamabe constant 
\begin{equation*}
Y(X,[g])< \Yl{X}
\end{equation*}
and if the scalar curvature of $g$ satisfies $S_g$ is in $L^{q}(X)$ for some $q > n/2$, then there exists $u$ bounded on $X$ in $\Sob{2}$ which attains $Y(X,[g])$ and such that $\tilde{g}=u^{\frac{4}{n-2}}g$ has constant scalar curvature. 

The main issue is that the explicit value of the local Yamabe constant is not known, even in the simplest case of conical singularities, simple edges or only codimension 2 singular strata. We are going to show that we can compute it in a large class of stratified spaces under a geometric assumption on the links.

\subsection*{Main Results} 
We consider a stratified space $(X,g)$ with Einstein links $(Z_j,k_j)$ of dimension $d_j$: by this condition we mean that the metric $k_j$ is such that $Ric_{k_j}=(d_j-1)k_j$ on the regular set of $Z_j$. We observe that this hypothesis on the Ricci tensor is well justified in view of proving the existence of a Yamabe metric. It is showed in \cite{ACM12} that, if the scalar curvature of each link is equal to $d_j(d_j-1)$, then we have the integrability condition on the scalar curvature $S_g$. This is obviously the case under our assumption. 

We show the following:

\begin{theorem*}
Let $(Z^d,k)$ be a stratified space endowed with a metric $k$ such that $Ric_k=(d-1)k$ on the regular set of $Z$. Then the Yamabe constant of $\R^{n-d}\times C(Z^d)$ is either equal to $Y_n$ or to:
\begin{equation*}
\left(\frac{\vol_k(Z)}{\vol(\s^d)} \right)^{\frac 2n}Y_n.
\end{equation*}
\end{theorem*}

This extends to the setting of stratified spaces an analogous result by J. Petean contained in \cite{Petean}, concerning the Yamabe constant of a cone over a smooth compact manifold $(M^n,g)$ with $Ric_g\geq (n-1)g$.

In order to prove this result, we need to distinguish two cases which depend on the strata of lowest codimension, i.e. equal to 2. The links of such strata must be circles $\s^1_a$ with radius $a$: remark that, when $a<1$, the cone $C(\s^1_a)$ is an Alexandrov space with positive curvature, in the sense of triangle comparison; when $a>1$ the cone is non-positively curved. We refer here to the definition of curvature bound given in the book by D. Burago, Y. Burago and S. Ivanov \cite{BBI}.

We first assume $(X,g)$ is compact stratified space which does not posses any codimension 2 strata with link $S^1_a$ for $a \geq 1$. If $(X,g)$ satisfies this condition, we call it an admissible stratified space. In this case, we are able to prove that a bound by below on the Ricci tensor leads to a bound by below for the spectrum of the Laplacian.

\begin{theorem*}
Let $(X,g)$ be an admissible stratified space such that $Ric_g \geq (n-1)g$. Then the first non-zero eigenvalue of the Laplacian $\Delta_g$ is greater or equal than $n$. 
\end{theorem*}

This is a generalization of Lichnerowicz theorem for smooth compact manifolds. Observe that K. Bacher and K-T. Sturm prove in \cite{BacherSturm} a version of Lichnerowicz theorem for spherical cones over a smooth compact manifold $(M^n,g)$, $Ric_g \geq (n-1)g$, by using a curvature-dimension condition in the sense of Sturm and Lott-Villani. Our theorem applies in general to cones over any admissible stratified space.

This spectral gap allows us to extend a result by D. Bakry contained in \cite{Bakry}: we prove the existence of a Sobolev inequality with explicit constants, and this gives in turns a lower bound for $Y(X,[g])$. Furthermore the lower bound is attained in the Einstein case, so that we are able to compute the Yamabe constant of an Einstein admissible stratified space.\\
In order to prove the previous results, we need to study the regularity of a solution to the Schrödinger equation $\Delta_gu=Vu$ for $V \in L^{\infty}(X)$ and of its gradient. A theorem contained in \cite{ACM14} states that such regularity depends on the spectral geometry of the links: more precisely, on the first eigenvalue of the Laplacian on the links. Combining this with our singular version of Lichnerowicz theorem, we are able to control the $L^{\infty}$-norm of the gradient $du$ away from a tubular neighbourhood of the singular set $\Sigma$. 

In the case of a link being $\s^1_a$ with $a \geq 1$, and then cone of angle bigger that $2\pi$, we cannot apply Lichnerowicz theorem, which does not even hold on $S^1_a$. We will study the isoperimetric profiles of $\R^{n-2}\times C(\s^1_a)$. Firstly, it is easy to prove by direct computations that the cone $C(\s^1_a)$ endowed with the metric $dr^2+(ar)^2d\theta^2$ can be found as the limit of Cartan-Hadamard surfaces, i.e. $\R^2$ endowed with a metric of negative sectional curvature. This gives an example of a stratified space arising as a limit of smooth manifolds. Moreover we show that $\R^{n-2}\times C(\s^1_a)$ has the same isoperimetric profiles as the Euclidean space $\R^n$, and that its isoperimetric domains are the Euclidean balls not intersecting the singular set $\R^{n-2}\times \{0\}$. This allows us to apply a classical argument by G.Talenti, then to find an optimal Sobolev inequality on $\R^{n-2}\times C(\s^1_a)$ and finally deduce that its Yamabe constant is equal to the one of the standard sphere $Y_n$. 

Knowing the local Yamabe constant $\Yl{X}$ opens further questions. We would like to know under which hypothesis the strict inequality $Y(X,[g]) < \Yl{X}$ holds, or what happens in the case of equality. As we recalled above, in the compact smooth case we know that for dimension $n \geq 6$ and if $(M^n,g)$ is not locally conformally flat, the the strict inequality $Y(M^n,[g])< Y_n$ follows from a local computation by means of test functions (see \cite{Aubin}). It may be possible to reproduce the same kind of technique on stratified spaces. Furthermore, in dimension $n=3,4,5$ and for locally conformally flat manifolds, the positive mass theorem holds. At present, it is not known whether an equivalent theorem may be proven in the setting of compact stratified spaces. \newline

\textbf{Acknowledgements:} I would like to thank Benoît Kloeckner, who suggested to use Ros Product Theorem, in the occasion of his geometry seminar at the Laboratoire de Mathématiques Jean Leray in Nantes. 

\section{Some technical tools}

We start by recalling some useful concepts about the geometry of a compact stratified space $(X^n,g)$. 

First of all, we cannot define the usual tangent space at any point of $x \in X$, in particular if $x$ belongs to the singular set $\Sigma$. We can nevertheless consider the \emph{tangent cone} at $x$: assume that $x \in X_{j}$ and take the Gromov-Hausdorff limit of the pointed metric spaces $(X, \eps^{-2}g, x)$ as $\eps$ tends to zero. We follow \cite{ACM14} in order to state that this limit is unique at any point and it is a cone $(C(S_x), dr^2+r^2h_x)$, where $S_x$ is the $j$-fold spherical suspension of the link $Z_j$. More precisely, if $\s^{j-1}$ is the canonical sphere of dimension $(j-1)$ we have:
\begin{align}
\label{tgsphere} S_x &= \left[O, \frac{\pi}{2}\right]\times \s^{j-1} \times Z_j \\
\label{h_met} h_x &= d\varphi^2+\sin^2\varphi g_{\s^{j-1}}+\cos^2\varphi k_j
\end{align}
We refer to $S_x$ as the \emph{tangent sphere} at $x$. Observe that $S_x$ is a stratified space of dimension $(n-1)$. 

The tangent cone $(C(S_x), dr^2+r^2h_x)$ is in fact isometric to the product $\R^j \times C(Z_j)$ with the metric $dy^2+d\tau^2+\tau^2k_j$. We can rewrite the Euclidean metric $dy^2$ in polar coordinates
\begin{equation*}
dy^2=d\rho^2+\rho^2g_{\s^{j-1}}.
\end{equation*}
where $g_{\s^{j-1}}$ is the canonical metric on the sphere of dimension $j-1$. Then with the change of coordinates $\tau= r\sin(\varphi)$, $\rho=r\cos(\varphi)$ we get $g= dr^2+r^2 h$ where $h$ coincides with the metric on the tangent sphere. This change of variables gives us the isometry we were looking for. 

Observe that, as a consequence, the local Yamabe constant of $(X,g)$  given in $\eqref{lYc}$ is also equal to:
\begin{equation*}
\Yl{X}= \inf_{x \in X} \{ Y(C(S_x), [dr^2+r^2h_x]) \}.
\end{equation*}

\subsection{Bounds by below for the Ricci tensor}

Let $(X,g)$ be a stratified space with strata $X_j$, $j=1, \ldots N$ and links $(Z_j,k_j)$. We give some result about the relations between Ricci bounds for the metric $g$, $k_j$ and the metrics $h_x$ on the tangent spheres $S_x$. 

We observe that through this paper a Ricci bound means that we have a classical Ricci bound on the regular set $\Omega$, where the metric $g$ is a smooth Riemannian metric and the Ricci tensor $Ric_g$ is defined in the usual way. There exist other approaches based on generalized lower Ricci bounds for metric measure spaces, as introduced by Sturm and Lott-Villani: see for example \cite{BacherSturm} for recent developpements in the subject. 

\begin{lemma} 
\label{RicciBound}
Let $X$ be a compact stratified space endowed with a metric $g$ such that the Ricci tensor is bounded by below, i.e. there exists a constant $c \in \R$ such that:
\begin{equation*}
Ric_g\geq cg \quad \mbox{ on } \Omega
\end{equation*}
Then for each point $x \in X$ the tangent cone has non-negative Ricci tensor. Furthermore, on each link $(Z_j, k_j)$ of dimension $d_j$ we have $Ric_{d_j} \geq (d_j-1)k_j$.
\end{lemma}

\begin{proof}
As we stated above, the tangent cone at $x \in X_j$ is the Gromov-Hausdorff limit of $(X, \eps^{-2}g, x)$ as $\eps$ goes to zero. Furthermore, the corvergence is uniform in $C^{\infty}$ away from the singular set $\Sigma$, so that as a consequence we have:
\begin{equation*}
Ric_{g_{\eps}}=Ric_{g}\geq c g = \eps^2 cg_{\eps} \quad \mbox{on } \Omega.
\end{equation*}
Then when we pass to the limit  as $\eps$ goes to zero the Ricci tensor of the limit metric $dr^2+r^2h_x$ must be non-negative. Now it is not difficult to see that this implies:
\begin{equation*}
Ric_{h_x} \geq (n-2)h_x
\end{equation*}
Recall that the metric $h_x$ has the form $\eqref{h_met}$, and then $k_j$ on the link $Z_j$ of the stratum $X_j$ must satisfy $Ric_{k_j}\geq (d_j-1)k_j$. For both of the last two bounds, we refer to the formulas for the Ricci tensor of warped products and doubly warped products contained in Chapter 3 of \cite{Petersen}.
\end{proof}

Viceversa, we can assume that we have a Ricci bound on the links:

\begin{lemma}
\label{ConfEquiv}
Let $(Z^d,k)$ be a compact stratified space such that $Ric_k\geq (d-1)k$. Consider the metric $g=dy^2+d\tau^2+\tau^2k$ on $(\R^{n-d-1}\times C(Z))$ and let $S$ be the $(n-d-2)$-fold spherical suspension of $Z$
\begin{equation*}
S= \left[0,\frac{\pi}{2}\right]\times \s^{n-d-2} \times Z.
\end{equation*}
endowed with the metric:
\begin{equation}
h=d\varphi^2+\sin^2\varphi g_{\s^{n-d-2}}+\cos^2\varphi k
\end{equation}
Then the cone metric $dr^2+r^2h$ on $C(S)=(0, \infty) \times S$ has non-negative Ricci tensor. Moreover, $(\R^{n-d-1}\times C(Z),g)$ is conformally equivalent to $C(S)=(0, \pi) \times S$ endowed with a metric $g_c$ such that $Ric_{g_c}\geq(n-1)g_c$. 
\end{lemma}

\begin{proof}
By recalling again \cite{Petersen}, Chapter 3, page 71, $Ric_k \geq (d-1)k$ implies that the metric $h$ defined in $\eqref{h_met}$ has Ricci tensor such that $Ric_h \geq (n-2)h$. As a consequence the Ricci tensor of $dr^2+r^2h$ is non-negative. Furthermore, we know that $(\R^{n-d-1}\times C(Z),g)$ is isometric to $(C(S), dr^2+r^2h)$, where $C(S)=(0, +\infty) \times S$. Now $dr^2+r^2 h$ is conformal to the product metric on $\R\times S$, which in turns is conformal to the metric:
\begin{equation*}
g_c=dt^2+\sin^2(t)h 
\end{equation*}
on $(0,\pi) \times S$. For this metric $Ric_{g_c}\geq (n-1)g_c$ holds: this concludes the proof. 
\end{proof}
adapt their argument
In the following we are going to give a bound from below for the Yamabe constant of a compact stratified space $(X^n,g)$ with Ricci tensor bounded by below $Ric_g \geq (n-1)g$. Thanks to the previous Lemma, we know that the tangent cones are conformal to a compact stratified space with this hypothesis on the Ricci tensor. As a consequence, we will also have a result about the \emph{local} Yamabe constant of a stratified spaces with links $(Z^d,k)$ such that $Ric_k\geq (d-1)k$.

\subsection{Regularity}

We recall that on a compact stratified space $(X,g)$ we can define $W^{1,p}(X)$ as the completion of the Lipschitz functions with the norm of $W^{1,p}(X)$. Following \cite{ACM12}, when $p$ is smaller or equal than the codimension $m$ of the singular set $\Sigma$, we assume that $C^1_0(\Omega)$ is dense in $W^{1,p}(X)$. From Proposition 2.2 in \cite{ACM12} we also know that the Sobolev inequality holds on $(X,g)$, i.e. there exist positive constants $A,B$ such that for any $u \in \Sob{2}$:
\begin{equation}
\label{Sob1}
\norm{u}_{\p}^2 \leq A\norm{u}^2_2+B\norm{du}_2^2.
\end{equation}
We are going to study the regularity of the gradient $du$ of a function $u \in \Sob{2}$ solving a Schrödinger equation of the form $\Delta_g u=Vu$, for $V \in L^{\infty}(X)$. We show that we can control the $L^{\infty}$-norm of the gradient $|du|$ on $\Omega$ depending on the distance from the singular set. \\
In order to do this, we need two hypothesis: the first one is that the Ricci tensor is bounded by below. The second one is a condition on the first eigenvalue $\lambda_1(S_x)$ of the Laplacian on the tangent spheres $S_x$ (or equivalently, on the links $Z_j$). 

\begin{prop}
\label{reg}
Let $(X^n,g)$ be a compact stratified space such that $Ric_g\geq(n-1)g$. Assume that for any $x \in X$ we have $\lambda_1(S_x) \geq (n-1)$. Let $u \in W^{1,2}(X)$ be a solution of:
\begin{equation}
\label{eqB}
\Delta_g u=V u 
\end{equation}
for $V \in L^{\infty}(X)$. Assume that there exists a constant $c$ such that $\Delta_g|du|\leq c|du|$. Then for any $\varepsilon>0$ we have:
\begin{equation}
\label{TubEst}
\norm{du}_{L^{\infty}(X\setminus \Sigma^{\varepsilon})} \leq C \sqrt{|\ln(\varepsilon)|}
\end{equation}
where $\tub{\eps}$ is an $\eps$-tubular neighbourhood of $\Sigma$ and $C$ is a positive constant not depending on $\eps$.
\end{prop}

\begin{rem}
\label{lp}
Since $|du|$ satisfies the estimate $\eqref{TubEst}$, it is in $L^p(X)$ for any $p \in [1, +\infty)$.  In fact, if we denote by $m$ the codimension of the singular set $\Sigma$, which is greater or equal to two, we have:
\begin{align*}
\int_X |du|^p dv_g & = \int_{X \setminus \tub{\eps}} |du|^p dv_g + \int_{\tub{\eps}} |du|^p dv_g \\
& \leq |\ln(\eps)|^{\frac p2}\vol_g(X)+ C^p\int_0^{\eps}\left(\int_{\partial\tub{t}}|\ln(t)|^{\frac{p}{2}} d\sigma_g \right) dt \\
& \leq |\ln(\eps)|^{\frac p2}\vol_g(X)+ C_1\int_0^{\eps} t^{m-1} |\ln(t)|^{\frac{p}{2}} dt.
\end{align*}
Where we used that the volume of boundary of the tubular neighbourhood of size $t$ is bounded by a constant times the $(m-1)$ power of $t$. The last integral is clearly finite, therefore $|du|\in L^{p}(X)$.
\end{rem}

The proof of Proposition $\ref{reg}$ consists of two steps: we state in the following the results we need to obtain it. 

\begin{lemma}[Moser iteration technique]
\label{L2}
Let $(X,g)$ be a compact stratified space and $f \in L^2(X)$ such that the inequality $\Delta_g f \leq c f$ holds on $\Omega$ for some positive constant $c$. Then there exists a constant $c_1$, only depending on the dimension $n$, such that for any $x \in \Omega$ and $0<r< d_g(x,\Sigma)/2$ we have:
\begin{equation*}
\norm{f}_{L^{\infty}(B(x,r/2))} \leq c_1 \left( \frac{1}{r^n} \int_{B(x,3r/4)} f^2 dv_g\right)^{\frac 12}.
\end{equation*}
where $c_1$ depends on $c$, on the dimension $n$ and on the constants appearing in the Sobolev inequality. 
\end{lemma}

\begin{proof}
We claim that if $\Delta_g f \leq c f$ on $\Omega$ then for any $\gamma >1$ we have:
\begin{equation}
\label{powers}
\Delta_g (f^{\gamma}) \leq c \gamma f^{\gamma} \quad \mbox{ on } \Omega.
\end{equation}
For any $\eps>0$ define $f_{\eps}= \sqrt{f^2+\eps^2}>0$. Consider the Laplacian of $f_{\eps}^2$ on $\Omega$:
\begin{equation*}
f_{\eps}\Delta_g f_{\eps}-|df_{\eps}|^2 = \frac{1}{2}\Delta_g(f_{\eps}^2)=f\Delta_g f-|df|^2 \leq c f^2-|df|^2 \leq c f_{\eps}^2 -|df_{\eps}|^2.
\end{equation*}
We have shown that $f_{\eps}\Delta_g f_{\eps} \leq c f_{\eps}^2$ on $\Omega$. Now for $\gamma>1$ consider $\Delta_g(f_{\eps}^{\gamma})$. Since $x^{\gamma}$ is a convex function, non-decreasing on $\R^{+}$, on $\Omega$ we have:
\begin{align*}
\Delta_g(f_{\eps}^{\gamma})
&= \gamma (f_{\eps}^{\gamma-1}\Delta_g f_{\eps} -(\gamma-1) f_{\eps}^{\gamma-2}|df_{\eps}|^2) \\
& \leq \gamma f_{\eps}^{\gamma-1}\Delta_g f_{\eps} \\
& \leq c \gamma f_{\eps}^{\gamma}.
\end{align*}
where in the last inequality we used the fact that $f_{\eps}\Delta_g f_{\eps} \leq c f_{\eps}^2$ on $\Omega$. Now if we let $\eps$ go to zero we obtain $\eqref{powers}$. \\
Now let $R_0=d_g(x, \Sigma)/2$ and choose $0<r<R<R_0$. Consider a smooth function $\varphi$ having compact support in $B(x, R_0)$ such that $\varphi$ is equal to one in $B(x,r)$, it vanishes outside $B(x,R)$ and its gradient satisfies:
\begin{equation*}
|d\varphi|\leq \frac{2}{(R-r)}
\end{equation*}
Let us consider $\varphi f$. Then we have:
\begin{align*}
\int_{B(x,R)}|d(\varphi f)|^2 dv_g & = \int_{B(x,R)} (|d\varphi|^2f^2+\varphi^2f \Delta_gf) dv_g \\
\leq & \int_{B(x,R)} (|d\varphi|^2f^2+c\varphi^2f^2)dv_g \\
\leq & \frac{A_1}{(R-r)^2}\int_{B(x,R)} f^2 dv_g.
\end{align*}
for some positive constant $A_1$. We then apply the Sobolev inequality $\eqref{Sob1}$ to $\varphi f$:
\begin{align*}
\left(\int_{B(x,R)} |\varphi f|^{\p} dv_g \right)^{\frac{n-2}{n}} 
& \leq  A \int_{B(x,R)}\varphi^2f^2 dv_g + B \int_{B(x,R)}|d(\varphi f)|^2dv_g \\
& \leq  A \int_{B(x,R)}\varphi^2f^2 dv_g + \frac{A_1}{(R-r)^2}\int_{B(x,R)} f^2 dv_g \\
& \leq\frac{A_2}{(R-r)^2}||f||^2_{L^2(B(x,R)}
\end{align*}
If we denote $\gamma=\frac{n}{n-2}$, we have shown that:
\begin{equation}
\label{M}
\norm{f}_{L^{2\gamma}(B(x,r))}\leq \left(\frac{A_2}{(R-r)^2}\right)^{\frac{1}{2}}\norm{f}_{L^2(B(x,R)}
\end{equation}
Consider for $j\in \mathbb{N}$ the sequence of radius
\begin{equation*}
\begin{split}
r_j&=\Big(\frac{1}{2}+2^{-(j+3)}\Big) R_0 \\
R_j&=\Big(\frac{1}{2}+2^{-(j+2)}\Big) R_0.
\end{split}
\end{equation*}
so that $R_j-r_j=2^{-j-3}R_0$ and $R_{j+1}=r_j$. Thanks to $\eqref{powers}$, we can apply the same argument we used for $\varphi f$ to $\varphi f^{\gamma}$, and so on iteratively with $\gamma^j$, for $j=1, \ldots N$. 
This leads to:
\begin{equation}
||f||_{L^{2\gamma^N}(B(x,r_N))}\leq \prod_{j=0}^{N-1}\left(\frac{2^{2(j+3)}A_2 \gamma^j}{R_0^2}\right)^{\frac{1}{2\gamma^j}}||f||_{L^{2}(B(x,3R_0/4))}
\end{equation}
When we let $N$ tend to $\infty$, the left-hand side converges to the $L^{\infty}$-norm of $f$ on $B(x,r)$ , and the product in the right-hand side converges to a constant $C$ divided by $R_0^{-n/2}$. In fact we have:
\begin{align*}
\ln\left(\prod_{j=0}^{N-1}\left(\frac{2^{2(j+3)}A_2\gamma^j}{R_0^2}\right)^{\frac{1}{2\gamma^j}}\right)=
\frac{\ln(2)}{2}\sum_{j=0}^{N-1}\frac{j+3}{\gamma^j}&+\frac{\ln (\gamma)}{2}\sum_{j=0}^{N-1}\frac{j}{\gamma^j}\\
&+\frac{1}{2}\ln \left(\frac{A_2}{R_0^2}\right)\sum_{j=0}^{N-1}\frac{1}{\gamma^j}.
\end{align*}
The first two sums converges to a constant as $N$ tends to infinity. The last one tends to $\displaystyle \frac{1}{1-1/\gamma}=\frac n2$ so that at the end we obtain: 
\begin{equation*}
\norm {f}_{L^{\infty}(B(x,R_0/2))}\leq c_1 \left( \frac{1}{R_0^n} \int_{B(x,3R_0/4)} f^2 dv_g\right)^{\frac 12}.
\end{equation*}
as we wished.
\end{proof}

We recall a result contained in \cite{ACM14}(see Theorem A and Proposition 4.1): it allows to study the regularity of solutions of the Schrödinger equation $\Delta_g u = Vu$, for $V \in L^{\infty}(X)$, depending on the geometry of the tangent spheres. 

\begin{prop}
\label{1.3APP}
Let $(X^n,g)$ be a compact stratified space and $u \in W^{1,2}(X)$ be a solution to
\begin{equation*}
\Delta_g u= Vu.
\end{equation*}
for $V \in L^{\infty}(X)$.
Assume that  for any $x \in X$ we have $\lambda_1(S_x)\geq n-1$. Then there exist a constant $C$ and a sufficiently small radius $r_0$ such that for any $x \in X$ and $0<r<r_0$ we have:
\begin{equation}
\frac{1}{r^n}\int_{B(x,r)}|du|^2 dv_g \leq C|\ln(r)|. 
\end{equation}
\end{prop}
\begin{rem} 
In Section 3.6 of \cite{ACM14}, it is shown that asking for any $x \in X$ that $\lambda_1(S_x)\geq (n-1)$ is equivalent to ask that for each link $(Z_j,k_j)$ the first eigenvalue $\lambda(Z_j)$ of the Laplacian with respect to $k_j$ is greater or equal than the dimension of $Z_j$.
\end{rem}
The proof of Proposition $\ref{reg}$ follows from the previous results:
\begin{proof}[Proof of Proposition $\ref{reg}$]
Let $x \in \Omega$ and $B(x,r)$ a ball of radius $0 < r < d_g(x, \Sigma)/2$, which is entirely contained in $\Omega$. Lemma $\ref{L2}$ allows us to bound the $L^{\infty}$-norm of $|du|$ on $B(x,r/2)$ with the mean of its $L^2$-norm on a ball of radius $3r/4$. The square of this last quantity is bounded by some constant times $|\ln(r)|$, thanks to Proposition $\ref{1.3APP}$. Therefore, we get the desired inequality outside an $\eps$ tubular neighbourhood of $\Sigma$ by choosing an appropriate small radius $r$.
\end{proof}

\subsubsection{Applications}
Assume that $(X^n, g)$ is a compact stratified spaces which satisfies the hypothesis of Proposition $\ref{reg}$, i.e. such that $Ric_g \geq (n-1)g$ and for any $x \in X$ we have $\lambda_1(S_x)=(n-1)$. We give two examples of equations to which we can apply Proposition $\ref{reg}$. \newline

\textbf{Example 1}: Let $\varphi$ be a locally Lipschitz function on $\R$ and consider $u \in W^{1,2}(X)$ solution of:
\begin{equation}
\label{ex}
\Delta_g u = c \varphi(u).
\end{equation}
Remark that among these solutions there are clearly the eigenfunctions of the Laplacian. It is possible to show that $u$ is bounded, by applying Moser iteration technique as in Proposition 1.8 in \cite{ACM12}. Moreover, it is not difficult to prove that there exists a constant $c_1$ such that on $\Omega$ we have $\Delta_g |du|\leq c_1 |du|$. This is done by using Bochner-Lichnerowicz formula, as the following lemma shows. 

\begin{lemma}[Bochner method]
\label{L1}
Let $(X^n,g)$ be a compact stratified space such that $Ric_g\geq(n-1)g$. Let $u \in W^{1,2}(X)$ be a solution of $\eqref{ex}$. Then on the regular set $\Omega$ we have:
\begin{equation*}
\Delta_g |du| \leq c_1 |du|
\end{equation*}
for some positive constant $c_1$.
\end{lemma}

\begin{proof}
Since $u$ is a solution to $\eqref{ex}$, we can assume that it is positive. As we observed above, $u$ is also bounded. For $\varepsilon>0$, let us introduce
\begin{equation*}
f_{\varepsilon}=\sqrt{|du|^2+\varepsilon^2}>0
\end{equation*}
We will consider $\Delta_g(f_{\varepsilon}^2)$ in order to obtain an inequality of the type $f_{\eps}\Delta_gf_{\varepsilon} \leq c f_{\varepsilon}^2$: dividing by $f_{\eps}$ and letting $\varepsilon$ tend to zero will allow us to conclude. We have
\begin{equation*}
f_{\varepsilon}\Delta_gf_{\varepsilon}-|df_{\varepsilon}|^2 =\frac{1}{2}\Delta_g(|du|^2+\varepsilon^2)=(\nabla^*\nabla du, du)-|\nabla du|^2.
\end{equation*}
The Bochner-Lichnerowicz formula holds on the regular set $\Omega$. By applying it to the equation $\eqref{ex}$ we get:
\begin{equation*}
\nabla^*\nabla du + Ric_g(du)=c\varphi'(u)du
\end{equation*}
We can now multiply by $du$. Since $u$ is bounded, the Ricci tensor $Ric_g$ is bounded by below by $(n-1)g$ and the derivative of $\varphi$ is bounded on $[0, \norm{u}_{\infty}]$, we obtain:
\begin{equation*}
(\nabla^*\nabla du, du) \leq c_1 |du|^2 -(n-1)|du|^2.
\end{equation*}
As a consequence we have:
\begin{equation*}
\begin{split}
(\nabla^*\nabla du, du) -|\nabla du|^2 
& \leq  c_1|du|^2 -(n-1)|du|^2 -|\nabla du|^2 \\
& \leq c_1|du|^2 - |\nabla du|^2.
\end{split}
\end{equation*}
We also observe that, by elementary calculations and Kato's inequality:
\begin{equation*}
|df_{\varepsilon}|^2=\frac{|du|^2|\nabla|du||^2}{|du|^2+\varepsilon^2}\leq |\nabla|du||^2 \leq |\nabla du|^2.
\end{equation*}
and as a consequence we get: 
\begin{align*}
f_{\varepsilon}\Delta_g f_{\varepsilon}-|df_{\varepsilon}|^2 & =
(\nabla^*\nabla du, du) -|\nabla du|^2 \\
&\leq c_1|du|^2 - |\nabla du|^2 \\
& \leq c_1 f_{\eps}^2 - |df_{\varepsilon}|^2.
\end{align*}
In conclusion $f_{\eps} \Delta_g f_{\varepsilon} \leq c_1 f_{\eps}^{2}$. Since $f_{\varepsilon}$ is positive everywhere, we can divide and obtain $\Delta_gf_{\varepsilon}\leq c f_{\varepsilon}$. By letting $\varepsilon$ go to zero, we deduce the desired inequality on $|du|$. 
\end{proof}
\textbf{Example 2}: 
When the metric is Einstein and the Ricci tensor $Ric_g$ is exactly equal to $(n-1)g$, we can also apply Proposition $\ref{reg}$ to the Yamabe equation:
\begin{equation}
\label{Ye}
\Delta_g u + a_n S_g u= a_nS_g u^{\frac{n+2}{n-2}}. 
\end{equation}
It is proven in \cite{ACM12} that a solution $u$ of the Yamabe equation $\eqref{Ye}$, when it exists, is in $W^{1,2}(X) \cap L^{\infty}(X)$. When $Ric_g=(n-1)g$, the scalar curvature is constant: thus we can differentiate the equation as we did in Lemma $\ref{L1}$ and apply the same technique in order to obtain $\Delta|du|\leq c_1 |du|$.

\section{Eigenvalues of the Laplacian Operator}
The aim of the next session is to show that the condition $\lambda_1(S_x)\geq (n-1)$ for any $x \in X$ holds in a large class of stratified spaces. Such class is given by admissible stratified spaces, that we define as follows:
\begin{D}
An \emph{admissible stratified space} is a compact stratified space $(X^n,g)$ which satisfies the following assumptions:
\begin{itemize}
\item[(1)] If there exists a stratum $X_{n-2}$ of codimension 2, its link has diameter smaller than $\pi$.
\item[(2)] The iterated edge metric $g$ satisfies $Ric_g\geq k(n-1)$, for some $k>0$, on the dense smooth set $\Omega$.
\end{itemize}
\end{D}
A classical result by Lichnerowicz states that for a compact smooth manifold $(M^n,g)$ with $Ric_g\geq k(n-1)$ with $k>0$, the lowest non-zero eigenvalue of the Laplacian is greater or equal to $kn$ (see for example \cite{Gallot}).We are going to extend this result to the case of admissible stratified spaces.

\begin{theorem}
\label{Lich}
Let $(X,g)$ be an admissible stratified space. Any non-zero eigenvalue $\lambda$ of the Laplacian $\Delta_g$ is greater or equal to $kn$.
\end{theorem}

\begin{rem}
For smooth Riemannian manifolds, a theorem by Obata characterizes the case of equality (see \cite{Obata}), i.e. under the same hypothesis of Lichnerowicz's theorem, $\lambda_1=kn$ if and only if $(M^n,g)$ is isometric to the standard sphere $\s^n$ of radius $1/\sqrt{k}$. In the case of stratified spaces, it would be interesting to obtain a similar result: we conjecture that if $(X^n,g)$ is an admissibile stratified space and the first non-zero eigenvalue of the Laplacian is equal to $kn$, then $(X^n,g)$ must be the spherical suspension of an $(n-1)$-dimensional admissible stratified space. 
\end{rem}

\begin{rem}
In \cite{BacherSturm} the authors give an analogous Lichnerowicz theorem for spherical cones $\Sigma(M)$ (considered as metric measure spaces) on a compact Riemannian manifold $(M,g)$ with lower Ricci bound $Ric_g \geq (n-1)g$. They use the existence of a curvature dimension condition $CD(n,n+1)$ on $\Sigma(M)$ in the generalized sense of Sturm and Lott-Villani. \\
Our theorem applies more generally to cones over  any stratified space $(X,g)$ having a lower Ricci bound on the regular set $\Omega$.
\end{rem}

\begin{proof}
Without loss of generality, we can rescale the metric and assume that $k=1$. We proceed by iteration on the dimension $n$ of the space. \\
If $n=1$, by our hypothesis $X$ must be a circle of diameter smaller than $\pi$. Then the first eigenvalue of the Laplacian is greater than $1$, and the proposition is true in dimension 1.

Assume that the statement is true for any dimension until $(n-1)$ and consider an admissible stratified space $X$ of dimension $n$. For any $x\in X$, the tangent sphere $S_x$ is an admissible stratified space of dimension $(n-1)$. By Lemma $\ref{RicciBound}$, the condition $Ric_g\geq (n-1)g$ implies that for any $x$ the metric $h_x$ satisfies $Ric_{h_x}\geq (n-2)h_x$. Therefore, by the iteration argument, for any $x \in X$:
\begin{equation*}
\lambda_1(S_x)\geq (n-1)
\end{equation*}

As a consequence, the hypothesis of Proposition $\ref{reg}$ are satisfied by $S_x$. As we have shown in the first example of 1.2.1 , we can apply this result to any eigenfunction $\varphi$ of the Laplacian. Therefore for any $\eps>0$ we have:
\begin{equation}
\norm{d\varphi}_{L^{\infty}(X\setminus \Sigma^{\eps})}\leq C\sqrt{|\ln(\eps)|}. 
\end{equation}
Recall also that by Moser iteration technique $\varphi$ is bounded.

Since we have an estimation of the behaviour of $d\varphi$ depending on the distance from the singular set, the rest of the proof is an adaptation of the classical one by means of well-chosen cut-off functions.
Consider for $\eps>0$ a cut-off function $\cutoff$, being equal to one outside $\tub{\eps}$, vanishing on some smaller tubular neighbourhood of $\Sigma$ and such that between the two tubular neighbourhoods $0 \leq \cutoff \leq 1$. We are going to specify the choice of such function in the following.

If $\varphi$ is an eigenfunction relative to the eigenvalue $\lambda$, then by the Bochner-Lichnerowicz formula on $\Omega$ we have:
\begin{equation*}
\nabla^*\nabla d\varphi + Ric_g(d\varphi)= \lambda d\varphi
\end{equation*} 
We then consider the Laplacian of $|d\varphi|^2$ and get:
\begin{equation}
\label{B1}
\frac 12 \Delta_g |d\varphi|^2= (\nabla^*\nabla d\varphi,d\varphi)-|\nabla d\varphi|^2 \leq \lambda |d\varphi|^2 -(n-1)|d\varphi|^2-|\nabla d\varphi|^2.
\end{equation}
If we multiply $\eqref{B1}$ by $\cutoff$ and integrate by parts we obtain:
\begin{equation}
\label{IntB1}
\int_X \Delta_g (\rho_{\varepsilon}) \frac{|d\varphi|^2}{2}dv_g \leq \int_X \rho_{\varepsilon} ((\lambda-(n-1))|d\varphi|^2-|\nabla d\varphi|^2) dv_g
\end{equation} 
We study the right-hand side and we consider the first term. By elementary calculations and integration by parts formula we can rewrite:
\begin{equation}
\label{IntB2}
\begin{split}
\int_X \rho_{\varepsilon}|d\varphi|^2dv_g&= \int_X(d(\cutoff\varphi),d\varphi)-\varphi(d\cutoff, d\varphi)) dv_g
 \\
& = \int_X \rho_{\varepsilon}\varphi\Delta_g\varphi dv_g - \int_X \varphi(d\cutoff, d\varphi)dv_g \\
& = \frac{1}{\lambda} \int_X \cutoff (\Delta_g\varphi)^2 dv_g - \int_X \varphi(d\cutoff, d\varphi)dv_g.
\end{split}
\end{equation}
In order conclude the proof, we need to choose $\cutoff$ such that when $\eps$ goes to zero we have:
\begin{itemize}
\item[(i)] the left-hand side of $\eqref{IntB1}$ tends to zero;
\item[(ii)] the last term of the right-hand side in $\eqref{IntB2}$ tends to zero.
\end{itemize}
If we can find such a cut-off function, when we pass to the limit as $\eps$ goes to zero we obtain:
\begin{equation*}
\left(1-\frac{(n-1)}{\lambda} \right)\int_X (\Delta_g\varphi)^2dv_g-\int_X|\nabla d\varphi|^2 dv_g\geq 0
\end{equation*}
Moreover, by Cauchy-Schwarz inequality $\displaystyle |\nabla du|^2 \geq \frac{(\Delta_g \varphi)^2}{n}$, so that we finally have:
\begin{equation*}
\left(1-\frac{(n-1)}{\lambda}-\frac 1n \right)\int_X (\Delta_g\varphi)^2dv_g \geq 0.
\end{equation*}
which leads to $\lambda \geq n$. \\
It remains to show that it is actually possible to construct a cut-off function having the properties (i) and (ii). This is done in the following. \\ 

\textbf{Choice of the cut-off functions} \\
We have to distinguish two different cases, whether the codimension $m$ of $\Sigma$ is strictly greater than two, or equal to two. \\

\emph{Case 1:} Firstly assume $m > 2$. Consider $\eps>0$ and the tubular neighbourhoods $\tub{\eps}$, $\tub{2\eps}$.
We want to build a cut-off function $\cutoff$ which is equal to 1 on $X \setminus \tub{2\eps}$ and vanishes on $\tub{\eps}$. Moreover, we need the gradient $d\cutoff$ and the Laplacian $\Delta_g \cutoff$ to decay "fast enough" as $\varepsilon$ tends to zero. We will obtain $\cutoff$ from a harmonic function, as explained in the following.\\
Let $h_{\varepsilon}$ be the harmonic extension of the function which is equal to $1$ on the boundary of $\Sigma^{2\varepsilon}$ and vanishes on the boundary of $\Sigma^{\varepsilon}$, i.e. $h_{\varepsilon}$ satisfies:
\begin{equation*}
\begin{cases}
\Delta_g h_{\varepsilon} =0 \\
h_{\varepsilon} = 1 \mbox{ on } \partial\Sigma^{2\varepsilon} \\
h_{\varepsilon} = 0 \mbox{ on } \partial \Sigma^{\varepsilon}.
\end{cases}
\end{equation*}
The harmonic extension has a variational characterization, i.e. if we consider the Dirichlet energy $\mathcal{E}$ defined by:
\begin{equation*}
\mathcal{E}(\varphi)= \int_{\Sigma^{2\varepsilon}\setminus \Sigma^{\varepsilon}} |d\varphi|^2 dv_g.
\end{equation*}
then $h_{\eps}$ attains the infimum of the functional $\mathcal{E}$ among all functions $\varphi \in W^{1,2}(X)$ taking values $1$ on $\partial\Sigma^{2\varepsilon}$ and vanishing on $\partial\Sigma^{\varepsilon}$.\\
Let $r$ be the distance function from the singular set $\Sigma$, i.e. $r(x)=d_g(x, \Sigma)$, and consider the following function $\psi_{\eps}$:
\begin{equation*}
\displaystyle \psi_{\varepsilon}(r)=
\begin{cases}
1 \mbox{ on } X \setminus \Sigma^{2\varepsilon} \\
\displaystyle \frac{r}{\varepsilon}-1 \mbox{ on }\Sigma^{2\varepsilon}\setminus \Sigma^{\varepsilon} \\
0 \mbox{ on } \Sigma^{\varepsilon}.
\end{cases}
\quad |d\psi_{\varepsilon}|= \frac{1}{\varepsilon}.
\end{equation*}
It is then easy to estimate the Dirichlet energy of $\psi_{\varepsilon}$:
\begin{equation*}
\mathcal{E}(\psi_{\varepsilon})= \int_{\Sigma^{2\varepsilon}\setminus \Sigma^{\varepsilon}}|d\psi_{\varepsilon}|^2dv_g \leq c' \varepsilon^{m-2}.
\end{equation*}
By the variational characterization of $h_{\varepsilon}$, $\mathcal{E}(h_{\varepsilon})\leq \mathcal{E}(\psi_{\varepsilon})$, so that
\begin{equation}
\label{hf1}
\mathcal{E}(h_{\varepsilon})\leq c' \varepsilon^{m-2}.
\end{equation}
However, $h_{\varepsilon}$ is not necessarily smooth. The cut-off function $\rho_{\varepsilon}$ will be obtained by composing $h_{\varepsilon}$ with a smooth function $\rho$ vanishing on $(-\infty,\frac{1}{4}]$ end being equal to one on $[\frac{3}{4},+\infty)$: more precisely, $\cutoff=\rho \circ h_{\eps}$. As a consequence we have:
\begin{equation*}
d\cutoff= (\rho'\circ h_{\eps}) dh_{\eps} \quad \text{and} \quad \Delta_g \cutoff= -(\rho''\circ h_{\eps})|dh_{\eps}|^2.
\end{equation*}
Since $\rho$ is smooth and chosen independently from $\eps$, there exist two constants $c_1, c_2$, not depending on $\varepsilon$, such that:
\begin{equation*}
|d\rho_{\varepsilon}| \leq c_1 |dh_{\varepsilon}|, \quad \text{and} \quad |\Delta \rho_{\varepsilon}| \leq c_2|dh_{\varepsilon}|^2.
\end{equation*}
We claim that our choice of $\cutoff$ satisfies (i) and (ii). For what concerns the first condition we obtain:
\begin{equation*}
\int_{\tub{2\eps}\setminus \tub{\eps}} |\Delta_g \cutoff||d\varphi|^2 dv_g \leq c_2 \int_{\tub{2\eps}\setminus \tub{\eps}} |dh_{\eps}|^2|d\varphi|^2dv_g \leq C_2|\ln(\eps)|\eps^{m-2}.
\end{equation*}
which tends to zero as $\eps$ goes to zero. As for the second condition (ii), by using Cauchy-Schwarz inequality twice and the estimate we have on $|d\cutoff|$, we get:
\begin{align*}
\int_{\tub{2\eps}\setminus \tub{\eps}} (d\cutoff,d\varphi)dv_g
& \leq \int_{\tub{2\eps}\setminus \tub{\eps}} |d\cutoff||d\varphi|dv_g \\
& \leq \left(\int_{\tub{2\eps}\setminus \tub{\eps}} |d\cutoff|^2 dv_g \right)^{\frac 12}
\left(\int_{\tub{2\eps}\setminus \tub{\eps}} |d\varphi|^2 dv_g \right)^{\frac 12} \\
& \leq c'_1 \eps^{\frac m2}\sqrt{|\ln(\eps)|} \left(\int_{\tub{2\eps}\setminus \tub{\eps}} |dh_{\eps}|^2 dv_g \right)^{\frac 12} \\
& \leq c^{''}_1 \eps^{m-1}\sqrt{|\ln(\eps)|}.
\end{align*}
which also tends to zero with $\eps$. \\

\emph{Case 2}: Consider $m=2$. The cut-off function $\cutoff$ will be equal to one outside $\tub{\eps}$ and it will vanish in $\Sigma^{\eps^2}$, for $0<\eps<1$. In this case too $\cutoff$ is obtained by "smoothing" the harmonic function $h_{\eps}$ being equal to one on $\partial \tub{\eps}$ and vanishing on $\partial \tub{\eps^2}$. We will be able to show that the Dirichlet energy of $\eps$ tends to zero when $\eps$ goes to zero as $|\ln(\eps)|^{-1}$. A priori this estimate does not suffices to show (i) and (ii), but only implies that the two integrals are bounded. For this reason we will need to give a more detailed study: we are going to prove that in fact $|d\varphi| \in W^{1,2}(X) \cap L^{\infty}(X)$.
Let $h_{\eps}$ be the harmonic function solving:
\begin{equation*}
\begin{cases}
\Delta_g h_{\varepsilon} =0 \\
h_{\varepsilon} = 1 \mbox{ on } \partial\Sigma^{\eps} \\
h_{\varepsilon} = 0 \mbox{ on } \partial \Sigma^{\eps^{2}}.
\end{cases}
\end{equation*}
We are going to exhibit a test function $f_{\eps}$ such that the Dirichlet energy $\mathcal{E}(f_{\eps})$ is bounded by a constant times $|\ln(\eps)^{-1}|$. Let $r$ be the distance function from $\Sigma$ as above. We define $f_{\eps}$:
\begin{equation*}
\displaystyle f_{\eps}(r)=
\begin{cases}
1 \quad \mbox{ on } X \setminus \tub{\varepsilon} \\
\displaystyle \left( 2- \frac{\ln(r)}{\ln(\eps)} \right) \quad \mbox{ on } \tub{\varepsilon}\setminus \tub{\eps^2} \\
0 \quad \mbox{ on } \tub{\eps^2}.
\end{cases}
\quad |df_{\varepsilon}|= \frac{1}{r|\ln(\eps)|}.
\end{equation*}
We claim that there exists a constant $A$ independent of $\eps$ such that:
\begin{equation}
\label{int1}
\left(\int_{\tub{\eps}\setminus \tub{\eps^2}}|df_{\varepsilon}|^2 dv_g \right)\leq \frac{A}{|\ln(\eps)|}.
\end{equation}
Let us assume that $-\ln(\eps)$ is an integer number $N$. Then we can decompose $\tub{\eps}\setminus \tub{\eps^2}$ in the disjoint union:
\begin{equation*}
\displaystyle 
\Sigma^{\varepsilon} \setminus \Sigma^{\varepsilon^2} = \bigcup_{j=N}^{2N-1} \Sigma^{e^{-j}}\setminus \Sigma^{e^{-(j+1)}}.
\end{equation*}
As a consequence the integral $\eqref{int1}$ can be written as the following sum:
\begin{align*}
\frac{1}{|\ln(\eps)|^2}\left(\int_{\tub{\eps}\setminus \tub{\eps^2}} \frac{1}{r^2} dv_g \right) 
& = \frac{1}{|\ln(\eps)|^2} \sum_{j=N}^{2N-1} \int_{\tub{e^{-j}}\setminus \tub{e^{-(j+1)}}}\frac{1}{r^2} dv_g \\
& \leq \frac{1}{|\ln(\eps)|^2} \sum_{j=N}^{2N-1} \int_{\tub{e^{-j}}\setminus \tub{e^{-(j+1)}}} e^{2(j+1)}dv_g \\
& \leq \frac{1}{|\ln(\eps)|^2} A (N-1) \leq \frac{A}{|\ln(\eps)|}.
\end{align*}
which is the estimate we wanted to prove. Then by the variational characterization of $h_{\eps}$ we have:
\begin{equation*}
\mathcal{E}(h_{\eps}) \leq \mathcal{E}(f_{\eps}) \leq \frac{A}{|\ln(\eps)|}.
\end{equation*}
Furthermore, thanks to our estimate on the behaviour of $d\varphi$ we obtain:
\begin{align*}
\int_{\tub{\eps}\setminus \tub{\eps^2}} |dh_{\eps}|^2 |d\varphi|^2 dv_g
& \leq C|\ln(\eps^2)|\int_{\tub{\eps}\setminus \tub{\eps^2}} |dh_{\eps}|^2dv_g \\
& \leq 2C^2|\ln(\eps)|\int_{\tub{\eps}\setminus \tub{\eps^2}} |df_{\eps}|^2dv_g \\
& \leq 2C^2 |\ln(\eps)| \frac{A}{|\ln(\eps)|} \leq B.
\end{align*}
where $B$ is a positive constant independent of $\eps$.

If we replace $\rho_{\eps}$ by $f_{\eps}$ in $\eqref{IntB1}$ and we let $\eps$ go to zero we then obtain a finite term $B_1$ on the left-hand side; therefore we obtain the following estimate:
\begin{equation*}
B_1 \leq \int_X ((\lambda -(n-1))|d\varphi|^2-|\nabla d\varphi|^2)dv_g.
\end{equation*}
Recall that $\varphi \in \Sob{2}$, so that the $L^{2}$-norm of $|d\varphi|$ is finite. Then the previous inequality tells us that also $|\nabla d\varphi|$ must be in $L^2(X)$, and so $\nabla|d\varphi|$ too, since we have clearly $|\nabla|d\varphi||\leq|\nabla d\varphi|$. As a consequence we have $|d\varphi| \in \Sob{2}$.  This allows us to get more regularity on $|d\varphi|$. \\

\textbf{Claim}: The gradient $d\varphi$ belongs to $L^{\infty}(X)$. 

\begin{proof}
Let us call $u=|d\varphi|$ for simplicity. We state that $u$ satisfies the weak inequality
\begin{equation}
\label{WeakIn}
\Delta_g u\leq cu.
\end{equation}
on the whole $X$. This means that for any $\psi \in W^{1,2}(X)$, $\psi \geq 0$ we have:
\begin{equation}
\label{WeakIn1}
\int_X (du, d\psi)_g dv_g \leq c\int_X u\psi dv_g. 
\end{equation}
We already proved that $\Delta_g u\leq c u$ strongly on $\Omega$, then we know that for any $\psi \in W^{1,2}(X)$ we have:
\begin{equation*}
\int_{\Omega} \psi \Delta_g u dv_g \leq c\int_{\Omega}u\psi dv_g.
\end{equation*}
In order to extend this inequality to the whole $X$ and obtain $\eqref{WeakIn}$, we consider $f_{\eps}$ defined as above, $0 \leq f_{\eps} \leq 1$ and we replace $\psi$ by $f_{\eps}\psi$. By integrating by parts we obtain:
\begin{equation}
\label{WeakIn2}
\int_X (d(f_{\eps}\psi), du)_g dv_g \leq c\int_X f_{\eps}\psi u + \int_X \psi(df_{\eps},du)_g dv_g. 
\end{equation}
We can use Cauchy-Schwarz inequality twice on the second term and obtain:
\begin{equation*}
\begin{split}
\int_X \psi(df_{\eps}, du)_g dv_g 
& \leq B_2 \norm{du}_2\left(\int_{X}|df_{\eps}|^2 dv_g\right)^{\frac 12} \\
& \leq B_3  \norm{du}_2 \frac{1}{\sqrt{|\ln(\eps)|}}.
\end{split}
\end{equation*}
Where we used the estimate $\eqref{int1}$ that we deduced above on the gradient $f_{\eps}$. Since the $L^2$-norm of the Hessian $du = \nabla d\varphi$ is finite, the second term in $\eqref{WeakIn2}$ tends to zero when $\eps$ goes to zero. Then letting $\eps$ go to zero in $\eqref{WeakIn2}$ implies $\eqref{WeakIn1}$, as we wished. 
Since $\eqref{WeakIn}$ is proven, Moser's iteration technique in Proposition 1.8 of \cite{ACM12} assures that $|d\varphi|\in L^{\infty}(X)$.
\end{proof} 

We are finally in the position to show that in codimension $m=2$ a cut-off functions satisfying (i) and (ii) exists: define $\rho_{\eps}=\rho \circ h_{\eps}$ for the same smooth function $\rho$ as before. We have for $c_1, c_2$ independent of $\eps$
\begin{equation*}
|d\cutoff|\leq c_1 |dh_{\eps}| \quad |\Delta_g \cutoff|\leq c_2 |dh_{\eps}|^2. 
\end{equation*}
The estimate on the Dirichlet energy on $h_{\eps}$ and the fact that the $L^{\infty}$-norm of $|d\varphi|$ is finite assures that $\cutoff$ is the desired cut-off function. For the condition (i) we obtain:
\begin{align*}
\int_X |\Delta_g \cutoff||d\varphi| dv_g \leq c'_2 \int_{X}|dh_{\eps}|^2 dv_g \leq \frac{c'_2A}{|\ln(\eps)|}.
\end{align*}
which tends to zero as $\eps$ goes to zero. For the condition (ii) we use Cauchy-Schwarz inequality twice 	and we get:
\begin{align*}
\int_X (d\cutoff, d\varphi)_g dv_g 
& \leq \left(\int_X |d\cutoff|^2 dv_g \right)^{\frac{1}{2}} \left(\int_X|d\varphi|^2\right)^{\frac{1}{2}}\\
& \leq c_1'\eps \left(\int_X |dh_{\eps}|^2 dv_g \right)^{\frac{1}{2}} \\
& \leq \frac{c_1^{''}\eps}{\sqrt{|\ln(\eps)|}}.
\end{align*}
which tends to zero as $\eps$ goes to zero. 

We have found an appropriate cut-off function for any codimension of the singular set $\Sigma$: this concludes the proof of the theorem.
\end{proof}

\begin{rem}
In the discussion above for the choice of the cut-off function in codimension $m>2$ (respectively $m=2$), we obtain $\cutoff$ by smoothing a harmonic function $h_{\eps}$ and by considering  $\rho \circ \psi_{\eps}$ (respectively $\rho \circ f_{\eps}$) because we need a condition on the Laplacian of $\cutoff$. The distance function from $\Sigma$ is not necessarily smooth: we know that almost everywhere $|dr|^2=1$, but we do not have information on the behaviour of its Laplacian. \end{rem}

\begin{rem}
By our definition of admissible stratified space, we are excluding the existence of a stratum  of codimension $2$ whose link is a circle $\s^1_a$, of radius $a$ bigger or equal to one. Recall that the classical Lichnerowicz theorem does not hold for $\s^1_a$, since the first eigenvalue of the Laplacian is equal to $1/a^2<1$: the first iterative step in our proof could not be applied. 
\end{rem}

\section{A bound by below for the Yamabe constant}

The following theorem is inspired by a result by Dominique Bakry in $\cite{Bakry}$, which gives a Sobolev inequality with an explicit constant on smooth compact Riemannian manifolds $(M,g)$ satisfying $Ric_g \geq k(n-1)$, $k>0$.

\begin{theorem}
\label{BakryS}
Let $X$ be an admissible stratified space of dimension $n$. Then for any $\displaystyle 1<p\leq 2n/(n-2)$ a Sobolev inequality of the following form holds:
\begin{equation}
\label{SobP}
V^{1-\frac{2}{p}}\norm{f}_p^2\leq \norm{f}_2^2+\frac{p-2}{nk}\norm{df}_2^2.
\end{equation}
where $V$ is the volume of $X$ with respect to the metric $g$.
\end{theorem}

The existence of such Sobolev inequality allows us to compute the Yamabe constant of a compact Einstein stratified space, as the following corollary states.

\begin{cor}
\label{YamabeConst}
The Yamabe constant of an admissible stratified space $X$ is bounded by below:
\begin{equation}
Y(X,[g])\geq \frac{nk(n-2)}{4}V^{\frac{2}{n}}
\end{equation}
In particular, if $g$ is an Einstein metric, we have equality.
\end{cor}

\begin{proof}
Recall that the Yamabe constant of $X$ is defined by
\begin{equation*}
 Y(X,[g])=\inf_{u \in W^{1,2}(X), u\neq 0} \frac{\displaystyle \int_X (|du|^2+a_n S_g u^2)dv_g}{\norm{u}_{\frac{2n}{n-2}}^2}.
\end{equation*}
where $a_n=\frac{n-2}{4(n-1)}$ and $S_g$ is the scalar curvature. Since $Ric_g\geq k(n-1)g$, we have $S_g\geq kn(n-1)$, and as a consequence 
\begin{equation*}
a_n S_g \geq \frac{nk(n-2)}{4}. 
\end{equation*}
We denote this constant by $\gamma^{-1}$. Remark that if we take $p=\frac{2n}{n-2}$ in the previous theorem, $\gamma$ is exactly the constant appearing in the right-hand side of the Sobolev inequality. Then for any $u\in W^{1,2}(X)$ we have:
\begin{equation*}
\frac{V^{\frac{2}{n}}}{\gamma}\norm{u}_{\frac{2n}{n-2}}^2\leq \norm{du}^2_2+\frac{1}{\gamma}\norm{u}_2^2 \leq \norm{du}^2_2 + \int_X a_nS_g |du|^2dv_g.
\end{equation*}
and this easily implies the desired bound by below on the Yamabe constant. 

Recall that an equivalent definition for the Yamabe constant is the following:
\begin{equation*}
\displaystyle Y(X,[g]) =\inf_{\tilde{g} \in [g]} Q(\tilde{g}), \quad Q(\tilde{g})=\frac{\int_X S_{\tilde{g}}dv_{\tilde{g}}}{\vol_{\tilde{g}}(X)^{1-\frac{2}{n}}}.
\end{equation*}
Where $[g]$ is the conformal class of $g$, consisting of all the metrics that can be written as $\tilde{g}=u^{\frac{4}{n-2}}g$ for some function $u\in C^{1}_0(\Omega)$. We call $Q(\tilde{g})$ the Yamabe quotient of $\tilde{g}$. When we consider an Einstein metric $g$ on an admissible stratified space, its Yamabe quotient attains exactly 
\begin{equation}
Q(g)=\frac{n(n-2)}{4}\vol_{g}(X)^\frac 2n
\end{equation}
since the scalar curvature of $g$ is constant and equal to $n(n-1)$. Thanks to our lower bound and the fact that the Yamabe constant is an infimum, we get the case of equality in the Einstein case. 
\end{proof}

We are now going to give the proof of theorem.

\begin{proof}
We can always rescale the metric in order to have $k=1$. By Theorem $\ref{Lich}$, we know that the first non-zero eigenvalue of the Laplacian is greater than $n$; moreover, as we recalled in Section 1, the Sobolev's inequality holds on $X$ (see Proposition 2.2 in \cite{ACM12}). From now on, we are using in our calculations the renormalized measure $d\mu=V^{-1}dv_g$, where $V=Vol_g(X)$.

The lower bound on the spectrum of the Laplacian, the Sobolev's inequality and Lemma 4.1 in \cite{Bakry} imply that there exists a positive constant $\gamma$ such that
\begin{equation*}
\norm{f}_{\frac{2n}{n-2}}^2\leq \norm{f}_2^2+\gamma \norm{df}_2^2.
\end{equation*}
By using interpolation between $2$ and $\frac{2n}{n-2}$, it is easy to see that for any $p<\frac{2n}{n-2}$ and for any $\delta>0$ we have the following inequality:
\begin{equation*}
\norm{f}_p^2 \leq (1+\delta)\norm{f}_2^2+\gamma_0\norm{df}_2^2
\end{equation*}
We denote by $\gamma_0$ the best constant appearing in the previous inequality. We are going to show that $\gamma_0$ is smaller then $(p-2)/n$, for any choice of $\delta>0$. By coming back to the measure $dv_g$, we will get the power $1-2/p$ of the volume and therefore the inequality $\eqref{SobP}$ will hold on X.\\
Consider a minimizing sequence for $\gamma_0$, i.e. a sequence of positive functions $(f_n)_n$ in $L^p(X)$ such that the quotient 
\begin{equation*}
\frac{\norm{f_n}_p^2-(1+\delta)\norm{f_n}_2^2}{\norm{df_n}_2^2}
\end{equation*}
converges to $\gamma_0$. We can assume without loss of generality that $||f_n||_2=1$. Then $(f_n)_n$ is bounded in $L^p(X)$ and by the compact embedding of $W^{1,2}(X)$ in $L^p(X)$ we can deduce that there exists a positive function $f$ in $W^{1,2}(X)$ such that $(f_n)_n$ converges weakly to $f$ in $W^{1,2}(X)$, and strongly in $L^p(X)$. Thanks to the normalization of the $L^2$-norm, $f$ is not vanishing everywhere, and thanks to the the choice of $\delta>0$, $f$ cannot be constant. Moreover, it satisfies the following equation on $X$:
\begin{equation}
\label{eqSob}
\gamma_0\Delta_g f+(1+\delta)f=f^{p-1}.
\end{equation}
We can apply the Moser iteration technique as in Proposition 1.8 in \cite{ACM12}, in order to show that $f$ is bounded. Since the Ricci tensor is bounded by below, we can apply the same technique we used in Lemma $\ref{L1}$ to show that $\Delta_g|df|$ is smaller or equal than $c|df|$ on $\Omega$ for some positive constant $c$. Furthermore, Theorem $\ref{Lich}$ assures that the condition $\lambda_1(S_x) \geq (n-1)$ holds for any $x \in X$, so that we can apply Proposition $\ref{reg}$ to $f$. Then, for any $\eps>0$ we have:
\begin{equation*}
\norm{df}_{L^{\infty}(X\setminus \Sigma^{\eps})} \leq C\sqrt{|\ln(\eps)|}.
\end{equation*}
We can express $f$ as the power of a function $u$, i.e. $f=u^{\alpha}$ for some $\alpha$ that will be chosen later. Then $u$ is also positive, bounded and its gradient satisfies the same estimate as $|df|$ away from a neighbourhood of the singular set $\Sigma$. \\
We can rewrite $\eqref{eqSob}$ in the form:
\begin{equation}
\label{eqSob1}
u^{\alpha(p-2)}=(1+\delta)+\gamma_0\frac{\Delta_g(u^{\alpha})}{u^{\alpha}}=(1+\delta)+\alpha \gamma_0 \left( \frac{\Delta_g u}{u}-(\alpha-1)\frac{|du|^2}{u^2} \right)
\end{equation}
Bakry's proof consists in multiplying this equation for an appropriate factor, and then by integrating it. He finds a factor depending on $\gamma_0^{-1}, p$ and $n$, multiplies by the $L^2$-norm of $du$, and he bounds it by below by some quantity, which is positive when $\alpha$ is well-chosen. We will proceed in a similar way, by taking care of introducing  a cut-off function, because we are allowed to use the equation $\eqref{eqSob1}$ and integration by parts only on the regular set $\Omega$. \\
If the codimension $m>2$ of $\Sigma$ is strictly greater than 2, consider the cut-off function $\rho_{\varepsilon}$ chosen in the proof of Theorem $\ref{Lich}$. If $m=2$ consider the function $f_{\eps}$ defined in the same proof. We must be careful with the codimension $m=2$, since we are still not sure that the Hessian of $u$ is in $L^2(X)$: we will be able to affirm it later in the proof. 
For $m>2$ we multiply $\eqref{eqSob1}$ by $\rho_{\varepsilon}u\Delta_g u$ and integrate on $X$; respectively for $m=2$ we multiply by $f_{\eps}u\Delta_g u$. For simplicity we write down the computations only for $\cutoff$: they are exactly the same for $f_{\eps}$. 
\begin{equation}
\label{Bakry1}
\begin{split}
\int_X \rho_{\varepsilon}u^{1+\alpha(p-2)}\Delta_g u d\mu &= 
(1+\delta)\int_X \rho_{\varepsilon}u\Delta_gu \\
&+\gamma_0\alpha\left(\int_X\rho_{\varepsilon}(\Delta_g u)^2d\mu 
-(\alpha-1)\int_X\rho_{\varepsilon}\frac{\Delta_g u}{u}|du|^2d\mu\right).
\end{split}
\end{equation}
When integrating by parts the left-hand side we obtain:
\begin{align*}
\int_X \rho_{\varepsilon}u^{1+\alpha(p-2)}\Delta_g u d\mu &= \int_X u^{1+\alpha(p-2)} (d\rho_{\varepsilon},du)_g d\mu \\
& + (1+\alpha(p-2))\int_X \rho_{\varepsilon}u^{\alpha(p-2)}|du|^2d\mu.
\end{align*}
Since $u$ is positive and bounded, we can bound $u^{1+\alpha(p-2)}$ by a positive constant independent of $\eps$. Then first term, which contains $(du,d\rho_{\varepsilon})_g$, tends to zero as $\eps$ goes to zero as we have shown in the proof of Theorem $\ref{Lich}$. This is true also for $m=2$ when we replace $\cutoff$ by the function $f_{\eps}$. In the second term we will replace $u^{\alpha(p-2)}$ by its value given by $\eqref{eqSob1}$. \\
As for the right-hand side of $\eqref{Bakry1}$, consider the first term:
\begin{equation*}
\int_X \rho_{\varepsilon}(u\Delta_gu) d\mu= \int_X u (du,d\rho_{\varepsilon})_g d\mu+\int_X \rho_{\varepsilon}|du|^2d\mu.
\end{equation*}
and when we let $\varepsilon$ tends to zero, since as before $u$ is bounded, we simply get the $L^2$-norm of $du$, both for the case $m >2$ and $m=2$.
\\
Therefore, after some elementary computation we obtain:
\begin{equation}
\label{ega}
\begin{split}
\frac{1+\delta}{\gamma_0}(p-2) \int_X \rho_{\varepsilon}|du|^2d\mu &= \int_X \rho_{\varepsilon} (\Delta_g u)^2d\mu \\
&+ (\alpha-1)(1+\alpha(p-2)) \int_X \cutoff \frac{|du|^4}{u^2}d\mu \\
&-\alpha(p-1) \int_X \cutoff \frac{\Delta_g u}{u}|du|^2d\mu + o(1).
\end{split}
\end{equation}
where we replaced the two terms containing $du$ and $d\cutoff$ by a term $o(1)$ which tends to zero as $\eps$ goes to zero. 
Let us denote:
\begin{align*}
I_1&= \int_X \rho_{\varepsilon} (\Delta_g u)^2dv_g. \\
I_2&= \int_X \cutoff \frac{\Delta_g u}{u}|du|^2dv_g.
\end{align*}
We are going to bound by below $I_1$ by integrating the Bochner-Lichnerowicz formula, which holds on the regular set $\Omega$, and to give an alternative expression for $I_2$ by integrating by parts. \\
Consider firstly $I_1$. We multiply the Bochner-Lichnerowicz formula
\begin{equation*}
(du,d\Delta_g u)_g=\Delta_g \frac{|du|^2}{2}+|\nabla du|^2+Ric_g(du,du) \mbox{ on } \Omega.
\end{equation*}
by the cut-off function $\cutoff$ and integrate. Recall that by hypothesis we have $Ric_g\geq (n-1)g$. \\
By rewriting $\cutoff(du,d\Delta_g u)_g= (du, d(\cutoff\Delta_gu))_g- \Delta_g u(du, d\cutoff)_g$ and integrating by parts, we then obtain:
\begin{equation}
\label{BakryI_1}
\begin{split}
\int_X\cutoff(\Delta_g u)^2d\mu \geq & \int_x \cutoff (|\nabla du|^2 dv_g+(n-1)|du|^2)d\mu 
\\ &  + \int_X \Delta_g\cutoff \frac{|du|^2}{2}d\mu +\int_X \Delta_g u (du,d\cutoff)_gd\mu.
\end{split}
\end{equation}
Remark that thanks to $\eqref{eqSob1}$, and the fact that $u$ is bounded, we know that $\Delta_g u$ can be split in the sum of a bounded term and a second term depending on $|du|^2$: it is equal to
\begin{equation*}
\Delta_g u= \frac{1}{\alpha}u\left(\frac{\Delta_g f}{u^{\alpha}}+\alpha(\alpha-1)\frac{|du|^2}{u^2}\right).
\end{equation*}
We know that $u$ is strictly positive and bounded, then the same is true for $u^{-1}$, and $\Delta_g f$ is bounded too. Furthermore, Remark $\ref{lp}$ in Section 1, assures that $|du|\in L^p(X)$, for all $p\in[1; +\infty)$. As a consequence $\Delta_g u$ also belongs to $L^{p}(X)$ for $p\in[1; +\infty)$. Then we can bound the last term in $\eqref{BakryI_1}$ by using Cauchy-Schwarz inequality:
\begin{equation*}
\int_X \cutoff\Delta_g u (du,d\cutoff)_gd\mu \leq \left(\int_X(\cutoff\Delta_gu)^2d\mu\right)^{\frac 12} \left(\int_X (du,d\cutoff)^2_gd\mu\right)^{\frac 12}. 
\end{equation*} 
where the first factor is finite, and the second one tends to zero as $\eps$ goes to zero. Then the last term in $\eqref{BakryI_1}$ tends to zero as $\eps$ goes to zero. 
As for the term containing $\Delta_g \cutoff$, we know that for $m>2$ this tends to zero as $\eps$ goes to zero. For $m=2$ and $f_{\eps}$, we only have that this quantity is bounded: but thanks to $\eqref{BakryI_1}$ this is enough to state that the $L^{2}$-norm of $\nabla du$ is finite. As in the proof of Theorem $\ref{Lich}$, this implies $|du|\in L^{\infty}(X)$ and there exists $\cutoff$, vanishing on $\tub{\eps^2}$, being equal to one away from $\tub{\eps}$, and $0 \leq \cutoff \leq 1$ on $\tub{\eps}\setminus \tub{\eps^2}$, such that the integrals of both $(du,d\cutoff)_g$ and $\Delta_g\cutoff |du|^2 $ tend to zero as $\eps$ goes to zero. From now on we are allowed to consider the cut-off function $\cutoff$ instead of $f_{\eps}$ also in the case of $m=2$. 

We can modify $\eqref{BakryI_1}$ a bit more. We decompose the Hessian $\nabla du$ in its traceless part $A$ plus $-(\Delta_g u / n)g$, since $\Delta_g u= -\mbox{tr}(\nabla du)$. Then the square norm of $\nabla du$ is equal to $|A|^2+(\Delta_gu)^2/n$, and therefore we get:
\begin{equation}
\label{I_1}
\int_X\cutoff(\Delta_g u)^2d\mu \geq \frac{n}{n-1}\int_X \cutoff |A|^2d\mu + n\int_X \cutoff |du|^2 d\mu + o(1).
\end{equation}
This will be the appropriate bound by below for $I_1$. \\
Now consider $I_2$ and integrate by parts:
\begin{equation*}
I_2= 2\int_X \cutoff\frac{\nabla du (du,du)}{u}d\mu
-\int_X \cutoff\frac{|du|^4}{u^2}d\mu
+ \int_X \frac{|du|^2}{u}(d\cutoff, du)_g d\mu.
\end{equation*}
With the same observations as before ($|du|\in L^{p}(X)$ for all $p \in[1+\infty)$ and Cauchy-Schwarz inequality), we can say that the last term in this expression tends to zero as $\eps$ goes to zero. We can decompose again the Hessian $\nabla du$ in $\nabla du = A -\frac{\Delta_gu}{n}g$. As a consequence we can write:
\begin{equation}
\label{I_2}
I_2 =
\frac{2n}{n+2} \int_X \cutoff\frac{A(du,du)}{u}d\mu-
\frac{n}{n+2} \int_X \cutoff \frac{|du|^4}{u^2}d\mu +o(1).
\end{equation}
We can now replace this expression for $I_2$ and the bound by below $\eqref{I_1}$ for $I_1$ in $\eqref{ega}$; after passing to the limit as $\varepsilon$ and $\delta$ go to zero we obtain:
\begin{equation}
\label{inéga}
\begin{split}
\left(\frac{1}{\gamma_0}(p-2)-n\right) \int_X |du|^2d\mu& \geq \frac{n}{n-1}\int_X |A|^2d\mu \\
& - \alpha(p-1)\frac{2n}{n+2}\int_X \frac{A(du,du)}{u}d\mu \\
& +C(\alpha) \int_X \frac{|du|^4}{u^2}d\mu.
\end{split}
\end{equation}
where 
\begin{equation*}
C(\alpha)=(\alpha-1)(1+\alpha(p-2))+\alpha(p-1)\frac{n}{n+2}.
\end{equation*}
The first two terms in the left-hand side of $\eqref{inéga}$ can be interpreted as a part of a square norm for some convenient coefficient: we can rewrite in fact
\begin{equation*}
\begin{split}
\left(\frac{1}{\gamma_0}(p-2)-n\right) \int_X |du|^2d\mu& \geq \frac{n}{n-1}\left(\int_X \left|A+\beta\frac{du\otimes du}{u}\right|^2d\mu \right) \\
&+\left(C(\alpha)-\beta^2 \frac{n}{n-1}\right) \int_X \frac{|du|^4}{u^2}d\mu.
\end{split}
\end{equation*}
where we have chosen:
\begin{equation*}
\beta= -\alpha(p-1)\frac{n-1}{n+2}
\end{equation*}
We denote by $T=\frac{du \otimes du}{u}$. Then, recalling that $A$ is traceless, we have
\begin{equation*}
|A+\beta T|^2\geq \frac{1}{n}\mbox{tr}(A+\beta T)^2=\frac{\beta^2}{n}\frac{|du|^4}{u^2}.
\end{equation*}
Replacing this in the previous inequality, we finally get:
\begin{equation}
\label{inéga1}
\left(\frac{1}{\gamma_0}(p-2)-n \right) \int_X |du|^2d\mu \geq (C(\alpha)-\beta^2)\int_X \frac{|du|^4}{u^2}d\mu.
\end{equation}
We remark that $C(\alpha)-\beta^2$ is a quadratic expression in $\alpha$. Its discriminant equals:
\begin{equation*}
-\frac{4n(p-1)((n-2)p-2n)}{(n+2)^2}
\end{equation*}
which is positive for $1<p<\frac{2n}{n-2}$. Therefore, thanks to our hypothesis, we can choose $\alpha$ in such a way that the right-hand side of $\eqref{inéga1}$ is a positive quantity. As a consequence we get for any $1<p< \frac{2n}{n-2}$:
\begin{equation*}
\frac{1}{\gamma_0} \geq \frac{n}{p-2}.
\end{equation*}
which gives the desired Sobolev inequality. We can pass to the limit as $p$ tends to $\frac{2n}{n-2}$ and get the result for $\frac{2n}{n-2}$ too.

\end{proof}

\subsection{Some examples}
Consider an admissible stratified space $(Z^d,k)$  of dimension $d$ with Einstein metric $k$. Thanks to Lemma $\ref{ConfEquiv}$, we know that $X=\R^{n-d-1}\times Z$ with the metric $g=dy^2+dr^2+r^2k$ is conformally equivalent to $(C(S),[dt^2+\cos^2(t)h])$, where:
\begin{align*}
S &=\left[0, \frac{\pi}{2} \right]\times \s^{n-d-2} \times Z \\
h &= d\varphi^2+\cos^2(\varphi)g_{\s^{n-d-2}}+\sin^2(\varphi)k.
\end{align*}
Moreover, $C(S)$ is an admissible stratified space endowed with an Einstein metric $g_c=dt^2+\cos^2(t)h$. Corollary $\ref{YamabeConst}$ states that its Yamabe constant will be equal to:
\begin{equation}
\label{Yc_ex}
Y(X,[g])=Y(C(S),[g_c])=\frac{n(n-2)}{4}\vol_{g_c}(C(S))^{\frac{n}{2}}.
\end{equation}
We claim that:

\begin{lemma}
\label{LYe}
Let $(Z^d,k)$ be an admissible stratified space of dimension $d$ with Einstein metric $k$. Then the Yamabe constant of $X=\R^{n-d-1}\times Z$ endowed with the metric $g$ as above is equal to:
\begin{equation}
Y(X,[g])= \left( \frac{\vol_k(Z)}{\vol(\s^{d})} \right)^{\frac{2}{n}} Y_n.
\end{equation}
\end{lemma}

\begin{proof}
Recall that the Yamabe constant of the sphere $Y_n$ is equal to 
\begin{equation*}
Y_n=\frac{n(n-2)}{4}\vol(\s^n)^{\frac 2n}.
\end{equation*}
Then by our expression above we have:
\begin{equation*}
Y(X,[g])=Y_n\left(\frac{\vol_{g_c}(C(S))}{\vol(\s^n)} \right)^{\frac{2}{n}}. 
\end{equation*}
Now the volume of $C(S)$ with respect to $g_c$ is clearly equal to:
\begin{equation*}
\vol_{g_c}(C(S))=2\vol_h(S) \int_{0}^{\frac{\pi}{2}}\cos^{n-1}(t)dt. 
\end{equation*}
and the volume of $S$ with respect to $h$ is:
\begin{equation*}
\vol_h(S)=\vol(\s^{n-d-3})\vol_k(Z)\int_{0}^{\frac{\pi}{2}} \cos^{n-d-2}(\varphi) \sin^d (\varphi) d\varphi.
\end{equation*}
By using polar coordinates, the sphere $\s^n$ can be viewed as the cone over the $(n-d-3)$-fold spherical suspension of $\s^d$, so that we can express its volume in the following form:
\begin{equation*}
\vol(\s^n)=2 \vol(\s^{n-d-3})\vol(\s^d)\int_{0}^{\frac{\pi}{2}} \cos^{n-d-2}(\varphi) \sin^d (\varphi) d\varphi \int_{0}^{\frac{\pi}{2}}\cos^{n-1}(t)dt.
\end{equation*}
Finally by replacing the two expressions for the volumes of $C(S)$ and $S^n$ we get the desired value of $Y(X,[g])$. 
\end{proof}

\textbf{Example}:
In the simplest case of $Z$ being a circle of radius $a<1$, $C(S^1_a)$ is a cone of angle $\alpha= 2\pi a$. A similar calculation leads to:
\begin{equation*}
Y(\R^{n-2}\times C(S^1_a), [g])= a^{\frac{2}{n}} Y_n=\left(\frac{\alpha}{2\pi}\right)^{\frac{2}{n}}Y_n.
\end{equation*}
Observe that this procedure cannot be applied if $Z$ is a circle with radius bigger than one, since we excluded the existence of codimension 2 strata with link of diameter bigger that $\pi$. In the next section we are going to give another way to compute the Yamabe constant of this kind of stratified spaces.

\begin{rem}
Lemma $\ref{LYe}$ extends a result by J. Petean about the Yamabe constant of cones. The author shows in \cite{Petean} that if $M$ is a compact manifold of dimension $n$, endowed with a Riemannian metric such that $Ric_g=(n-1)g$, then the Yamabe constant of the cone $C(M)=(0,\pi)\times M$ endowed with the cone metric $dt^2+\sin^2(t)g$ is equal to:
\begin{equation*}
Y(C(M), [dt^2+\sin^2(t)g])=\left(\frac{\vol_g(M)}{\vol(\s^n)}\right)^{\frac{2}{n+1}}Y_{n+1}.
\end{equation*}
If the spherical suspension $S$ were a compact smooth manifold, our computation would give exactly the same result. 
\end{rem}

\section{Cones of angle $\alpha \geq 2\pi$}
Let $(X^n,g)$ be a stratified space with one singular stratum $X_{n-2}$ of codimension $2$: we assume that its link is the circle $\s^1_a$ of radius $a>1$ and then that the cone $C(\s^1_a)$ has angle $\alpha=a2\pi\geq 2\pi$. Such stratified space does not belong to the class of admissible stratified spaces we defined above, and Theorem $\ref{Lich}$ does not hold on it. As a consequence, we cannot apply Corollary $\ref{YamabeConst}$ in order to compute its local Yamabe constant, or, equivalently, the Yamabe constant of $\R^{n-2}\times C(\s^1_a)$. \\
We are going to follow another strategy: we will study the isoperimetric profile of $X=\R^{n-2}\times C(\s^1_a)$, i.e. given a metric $g$ on $X$ we study the function $I_g$:
\begin{equation*}
I_g(v)=\inf \{ \vol_g(\partial E), \mbox{ \\ } E \subset X, \vol_g(E) =v \}.
\end{equation*}
An Euclidean isoperimetric inequality holds on $X$ if there exists a positive constant $c$ such that 
\begin{equation}
\label{IsopIn0}
I(v)\geq c v^{1-\frac{1}{n}}
\end{equation}
In the Euclidean space $\R^n$, the constant $c$ is given by the isoperimetric quotient of euclidean balls. Moreover, it is a well known result that the isoperimetric inequality in $\R^n$ is equivalent to the following Sobolev inequality: for any $n>1$ and $f\in W^{1,1}(\R^n)$
\begin{equation*}
\norm{f}_{q}\leq C \norm{df}_1, \quad q=\frac{n}{n-1}.
\end{equation*}
It is also possible to compute the explicit value for the optimal constant appearing in this inequality. Moreover, this inequality leads to the sharp inequalities for $1 \leq p <n$ see for example $\cite{Talenti}$):
\begin{equation}
\norm{f}_q \leq C_{n,p} \norm{df}_p, \quad q=\frac{np}{n-p}
\end{equation}
In the following, we are going to show that the isoperimetric profile of $X$ (with the appropriate metric) coincides with the isoperimetric profile of $\R^n$. This will give in turn a sharp Sobolev inequality and then the value of the Yamabe constant of $X=\R^{n-2}\times C(\s^1_a)$, for $a\geq 1$.

\subsection{Approaching $C(\s^1_a)$ with Cartan-Hadamard manifolds}
We are going to find a metric $h_{\eps}$ on $\R^2$, conformal to the Euclidean metric, that converges to a metric $h$ on $\R^2$ with one conical singularity, which is in turn isometric to $C(S^1_a)$ endowed with the metric $dr^2+(ar)^2d\theta^2$.

\begin{lemma}
\label{SeqMet1}
There exists a sequence of metrics $h_{\varepsilon}$ on $\R^2$, conformal to the Euclidean metric, with negative sectional curvature, such that $h_{\varepsilon}$ converges uniformly on any compact domain of $\R^2 \setminus \{ 0\}$ to the cone metric on $C(\s^1_a)$ with $a\geq 1$.
\end{lemma}

\begin{proof}
Consider the following metric on $\R^2$: 
\begin{equation}
h_{\varepsilon}=(\varepsilon^2+\rho^2)^{a-1}(d\rho^2+\rho^2d\theta^2)
\end{equation}
We can compute its sectional curvature $\kappa_{\varepsilon}$ by applying the formulas for conformal changes of metrics (see for example \cite{Besse}): 
\begin{equation*}
\begin{split}
g = & e^{2f_{\eps}}(d\rho^2+\rho^2d\theta^2), \quad f_{\eps}=\frac{a-1}{2}\ln(\rho^2+\eps^2) \\
\kappa_{\varepsilon} & = e^{-2f_{\eps}}\Delta_g f_{\eps} = -\frac{2(a-1)\rho}{(\rho^2+\eps^2)^{a+1}}.
\end{split}
\end{equation*}
which is non-positive, since by hypothesis $a \geq 1$. When $\varepsilon$ tends to zero, the conformal factor $(\rho^2+\varepsilon^2)^{a-1}$ converges to $\rho^{2(a-1)}$ pointwise and uniformly in $C^{\infty}$ on any compact domain. As a consequence $h_{\varepsilon}(\rho,\theta)$ converges to 
\begin{equation*}
h(\rho,\theta)=\rho^{2(a-1)}(d\rho^2+\rho^2d\theta^2)
\end{equation*}
which is a Riemannian metric on $\R^2 \setminus \{\rho=0\}$. Now, $\R^2$ endowed with the metric $h$ is a surface with one conical singularity $0$, which is isometric to $C(\s^1_a)$ endowed with the metric $dr^2+(ar)^2d\theta^2$: it suffices to apply the change of variables $r=\rho^a/a$.
\end{proof}

A Cartan-Hadamard manifold is a complete, simply connected Riemann manifold with nonpositive sectional curvatures. The following conjecture is known as the Cartan-Hadamard conjecture or Aubin's conjecture (see for example $\cite{Rit}$ ):
\begin{conj}
Let $(M^n,g)$ be a Cartan-Hadamard manifold, whose sectional curvatures satisfy $\kappa \leq c\leq 0$. Then the isoperimetric profile $I_M$ of $M^n$ is bounded from below by the isoperimetric profile $I_c$ of the complete and simply connected space $M^n_c$, whose sectional curvatures are equal to $c$.
\end{conj}
This conjecture has been proved in dimension $n=2, 3, 4$ by A. Weil \cite{Weil}, C. Croke  \cite{Croke} and B. Kleiner \cite{Kleiner}. In our particular case, $(\R^2, h_{\varepsilon})$ is a Cartan-Hadamard manifold with $c=0$. As a consequence we have:

\begin{lemma}
\label{CH}
Let $h_{\eps}$ be the metric on $\R^2$ defined in the previous lemma. Then the isoperimetric profile $I_{h_{\eps}}$ of $(\R^2,h_{\eps})$ is bounded by below by the isoperimetric profile $I_2$ of $\R^2$ with the Euclidean metric.
\end{lemma}

\subsection{Isoperimetric profiles}
We recall a result known in the literature as Ros Product Theorem and contained in $\cite{Ros}$, about the isoperimetric profiles of Riemannian products.
\begin{prop}[Ros Product Theorem]
Consider two Riemannian manifolds $(M_1,g_1)$ and $(M_2,g_2)$, $\mbox{dim}(M_2)=n$. Assume that the isoperimetric profile $I_2$ of $(M_2,g_2)$ is bounded by below by the isoperimetric profile $I_n$ of $\R^n$. Then the isoperimetric profile of the Riemannian product $(M_1\times M_2, g_1 + g_2)$ is bounded by below by the one of $(M_1\times \R^n, g_1 + \xi)$, where $\xi$ is the Euclidean metric. 
\end{prop}
The idea of the proof is to define an appropriate symmetrization for sets $E \subset M_1 \times M_2$.
Denote for simplicity $g=g_1+g_2$ and $g_0=g_1+\xi$. We consider for $x \in M_1$ the section $E(x)= E \cap ( \{x\}\times M_2)$. Then the symmetrization $E^s \subset M_1 \times \R^n$ will be the set defined by:
\begin{enumerate}
\item if $E(x)=\emptyset$, then $E^s(x)=\emptyset$;
\item if $E(x) \neq \emptyset$, then $E^s(x)=\{x\}\times B_r$, where $B_r$ is an euclidean ball in $\R^n$ such that $\vol_{\xi}(B_r)=\vol_{g_2}(E(x))$.
\end{enumerate}
By following Proposition $3.6$ in $\cite{Ros}$, $E^s$ satisfies that $\vol_{g_0}(E^s)=\vol_{g}(E)$ and $\vol_{g_0}(\partial E^s) \leq \vol_{g}(\partial E)$. This is enough to show that if $F \subset M_1\times M_2$ realizes the infimum in $I_{g}(v)$, i.e $\vol_g(F)=v$ and $\vol_g(\partial F)=I_g(v)$, then its symmetrization $F^s$ satisfies $\vol_{g_0}(F^s)=v$ and 
\begin{equation*}
I_{g_0}(v) \leq \vol_{g_0}(\partial F^s) \leq \vol_g(\partial F)
\end{equation*}
As a consequence, $I_g(v) \geq I_{g_0}(v)$ for any $v>0$.

\begin{prop}
\label{P_Isop}
Let $X=\R^{n-2}\times C(S^1_a)$ and denote by $g$ the metric $\xi+h$. Let $I_g$ be its isoperimetric profile. Then $I_g$ is coincides with the isoperimetric profile $I_n$ of $\R^n$ with the Euclidean metric.
\end{prop}

\begin{proof}
We will show firstly that $I_g$ is bounded by below by $I_n$. Consider $\R^{n-2}\times\R^2$ endowed with the metric $g_{\varepsilon}=\xi+ h_{\varepsilon}$, and denote by $I_{\varepsilon}$ the isometric profile with respect to this metric. Thanks to Lemma $\ref{CH}$ and to Ros Product theorem
we deduce that $I_{\varepsilon}$ is bounded by below by the isoperimetric profile of $\R^{n-2}\times \R^2$ with the euclidean metric, i.e. $I_n$. \\
Therefore we have for any bounded domain $E \subset X$, $\vol_{g_{\varepsilon}}(\Omega)=v$, with smooth boundary $\partial E$:
\begin{equation}
\label{isopEps}
\frac{\vol_{g_{\varepsilon}} (\partial E)}{\vol_{g_{\varepsilon}}(E)^{1-\frac{1}{n}}}\geq \frac{I_{\varepsilon}(v)}{v^{1-\frac{1}{n}}} \geq c_n
\end{equation}
where $c_n$ is the optimal constant appearing in the isoperimetric inequality in $\R^n$. When we pass to the limit as $\eps$ tends to zero, the volumes of both $E$ and $\partial E$ with respect to $g_{\eps}$ converge to the volumes with respect to $g$. 

In face, if we denote by $dx$ the $n$-dimensional Lebesgue measure on $\R^n$ and by $d\sigma$ the volume element induced on $\partial E$ by the Euclidean metric, we have for the volume of $E$:
\begin{equation*}
\lim_{\eps \rightarrow 0}\vol_{g_{\eps}}(E)= \int_{E} (\rho^2+\eps^2)^{(a-1)} dx =\vol_g (E)
\end{equation*}
since $(\rho^2+\eps^2)^{(a-1)}$ converges to $\rho^{2(a-1)}$ on any bounded domain. As for the volume of $\partial E$ we get:
\begin{align*}
\lim_{\eps \rightarrow 0}\vol_{g_{\eps}}(\partial E) &= 
\lim_{\eps \rightarrow 0} \int_{\partial E}(\rho^2+\eps^2)^{(a-1)} d\sigma \\
&= \lim_{\eps \rightarrow 0} \int_{\partial E \setminus\R^{n-2} \times  \{0\}}(\rho^2+\eps^2)^{(a-1)} d\sigma \\
&=  \int_{\partial E \setminus \R^{n-2} \times\{0\}}\rho^{2(a-1)} d\sigma \\
&= \int_{\partial E} \rho^{2(a-1)}d\sigma=\vol_g(\partial E).
\end{align*}
where we used again the convergence of the conformal factor and the fact that $\R^{n-2}\times \{0 \}$ has zero $(n-1)$-dimensional Lebesgue measure. 

Therefore when we pass to the limit as $\eps$ goes to zero in $\ref{isopEps}$ we obtain:
\begin{equation}
\label{isopIn}
\frac{\vol_{g} (\partial E)}{\vol_{g}(E)^{1-\frac{1}{n}}} \geq c_n
\end{equation}
Observe that $\R^{n-2}\times C^1(S_a)$ contains euclidean balls: they are the geodesic balls $\mathbb{B}^n$ not intersecting the singular set $\R^{n-2}\times \{ 0\}$. They realize $c_n$, so that for any $v>0$ the infimum defining $I_g(v)$ is attained by the euclidean geodesic ball of volume $v$, i.e. $I(v)=c_n v^{1-\frac{1}{n}}$. As a consequence, the isoperimetric profile $I_g$ coincides with $I_n$.
\end{proof}

\subsection{Yamabe constant of $\R^{n-2} \times C(S^1_a)$}

We have found an optimal constant for the isoperimetric inequality $\eqref{isopIn}$ with respect to a metric $g=\xi+h$ on $X=\R^{n-2}\times C(S^1_a)$. Such metric is isometric to $\xi+dr^2+(ar)^2d\theta^2$ on $X$, so they obviously define the same conformal class. As a consequence, we can compute the Yamabe constant of $\R^{n-2}\times C(S^1_a)$, as the following proposition shows.

\begin{prop}
The Yamabe constant of $X=\R^{n-2}\times C(S^1_a)$, $a>1$, is equal to the Yamabe constant $Y_n$ of the standard sphere of dimension $n$.
\end{prop}

\begin{proof}
In the Euclidean space $\R^n$, the existence of the isoperimetric inequality leads to the existence of a sharp Sobolev inequality: for any $1<p<n$ and for any $f \in W^{1,p}(\mathbb{R}^n)$:
\begin{equation}
\label{optSob}
\norm{f}_{q}\leq C_{n,p}\norm{df}_p, \quad q=\frac{np}{n-p}.
\end{equation}
The constant $C_{n,p}$ is optimal in the sense that it attains:
\begin{equation}
\label{ratio}
\displaystyle C_{n,p}^{-1}=\inf_{\underset{f \neq 0}{f\in W^{1,p}(\R^n)}} \frac{\norm{du}_p}{\norm{u}_q}.
\end{equation}
We briefly recall the ideas of the proof given by G. Talenti in $\cite{Talenti}$. For any Lipschitz function $u$ we can define the symmetrization $u^*$ in the following way: for any $t \in \R$, the level sets $E^*_t=\{x \in \R^n : u^*(x)>t\}$ of $u^*$ are Euclidean $n$-balls having the same volume as the level sets $E_t$ of $u$. Then $u$ is spherically symmetric and Lipschitz. It is possible to show that this kind of symmetrization makes the ratio $\eqref{ratio}$ decrease: from Lemma 1 in $\cite{Talenti}$ we have that for any $1<p<n$
\begin{equation*}
\norm{u}_q=\norm{u^*}_q \quad \mbox{ and } \quad \norm{du}_p\geq \norm{du^*}_p.
\end{equation*}
The first equality is trivial. The second inequality is deduced by using isoperimetric inequality and coarea formula, which relates the integral of $|du|$ with the $(n-1)$-measure of the boundaries $\partial E_t$ of level sets. \\
As a consequence, the infimum in $\eqref{ratio}$ is attained by spherically symmetric functions. Classical argument in the calculus of variations allows to prove that there exists a minimizer. Moreover, G. Talenti exhibits a family of functions attaining $C_{n,p}$ and gives its exact value. \\
When $p=2$, $(C_{n,2})^{-2}$ coincides with the Yamabe constant $Y_n$ of the $n$-dimensional sphere. This is shown by pulling back the functions attaining $C_{n,2}$ from $\R^n$ to the sphere $\s^n$ without the north pole.\\ 
In our case, $X=\R^{n-2}\times C^(\s^1_a)$ is flat and satisfies the Euclidean isoperimetric inequality $\eqref{isopIn}$. We can then repeat the same argument as Talenti to deduce that the inequality $\eqref{optSob}$ holds on $X$ with the same optimal constant as in $\R^n$. Furthermore, by definition of the Yamabe constant, and since $S_g=0$, we have:
\begin{equation*}
\displaystyle Y(X,[g])=\inf_{\underset{u>0}{u\in W^{1,2}(X)}}\frac{\int_X|du|^2 dv_g}{\norm{u}_{\frac{2n}{n-2}}^2}.
\end{equation*}
so that $Y(X,[g])$ is equal to $(C_{n,2})^{-2}$. We have then proved $Y(X,[g])=Y_n$. 
\end{proof}

\section{Conclusion}

Our results allows us to compute the Yamabe constant of an Einstein admissible stratified space, as Corollary $\ref{YamabeConst}$ states. Moreover, we can deduce from them an explicit value for the local Yamabe constant of a stratified space whose links are endowed with an Einstein metric. 

Let $(X,g)$ be a compact stratified space with strata $X_j$, $j=1  \ldots n$. Assume that each of its links $Z_j$ admits an Einstein metric $k_j$ such that
\begin{equation*}
Ric_{k_j}=(d_j-1)k_j
\end{equation*}
where $d_j$ is the dimension of $Z_j$. We have two possibilities: either $Z_j$ is an admissible stratified space, or it is a circle of radius $a\geq 1$. In both cases we are able to compute the Yamabe constant of $\R^{n-d_j-1}\times Z_j$. This leads to the following:

\begin{prop}
Let $(X,g)$ be a compact stratified space with strata $X_j$, $j=1, \ldots N$. Assume that each link $Z_j$ of dimension $d_j$ is endowed  with an Einstein metric $k_j$, such that $Ric_{k_j}=(d_j-1)k_j$. Then the local Yamabe constant of $X$ is given by:
\begin{equation*}
\Yl{X}= \inf\left\{ Y_n,  \left(\frac{\vol_{k_1}(Z_1)}{\vol(\s^{d_1})}\right)^{\frac{2}{n}}Y_n, \ldots \left( \frac{\vol_{k_N}(Z_N)}{\vol(\s^{d_N})}\right)^{\frac{2}{n}}Y_n \right\}.
\end{equation*}
\end{prop}

\bibliographystyle{amsalpha}
\bibliography{biblio}

\end{document}